\pdfoutput=1
\documentclass{amsart}

\usepackage{amssymb,mathtools,thmtools,thm-restate,graphicx,color,hyperref,dsfont,enumitem,tikz-cd, tikz}
\usepackage[utf8]{inputenc}
\usepackage[nameinlink,nosort]{cleveref}
\usepackage[final]{microtype}
\usepackage[margin=1.3in]{geometry}

\usetikzlibrary{knots, hobby}

\let\cref\Cref
\crefname{prop}{Proposition}{Propositions}
\crefname{thm}{Theorem}{Theorems}
\crefname{lem}{Lemma}{Lemmas}
\crefname{equation}{}{}
\hypersetup{colorlinks={true},linkcolor={black},citecolor={black},filecolor={black},urlcolor={black}}

\numberwithin{equation}{section}

\newtheorem{thm}{Theorem}[section]
\newtheorem{prop}[thm]{Proposition}
\newtheorem{lem}[thm]{Lemma}
\newtheorem{cor}[thm]{Corollary}

\newtheorem{q}[thm]{Question}

\theoremstyle{definition}

\theoremstyle{remark}
\newtheorem{rem}[thm]{Remark}
\newtheorem{exmp}[thm]{Example}

\renewcommand{\hom}{\operatorname{Hom}}
\newcommand{\ohom}{\operatorname{\ol{Hom}}}
\renewcommand{\ker}{\operatorname{Ker}}

\renewcommand{\top}{\operatorname{top}}

\newcommand{\N}{\mathbb{N}}
\newcommand{\Z}{\mathbb{Z}}
\newcommand{\Q}{\mathbb{Q}}
\newcommand{\R}{\mathbb{R}}
\newcommand{\C}{\mathbb{C}}
\newcommand{\F}{\mathbb{F}}

\newcommand{\ol}{\overline}

\DeclareMathOperator{\ext}{Ext}
\DeclareMathOperator{\GL}{GL}
\DeclareMathOperator{\im}{Im}

\DeclareMathOperator{\rank}{rank}
\DeclareMathOperator{\res}{Res}

\DeclareMathOperator{\tor}{Tor}

\DeclareMathOperator{\Bl}{Bl}
\DeclareMathOperator{\Tor}{Tor}
\DeclareMathOperator{\Ext}{Ext}
\DeclareMathOperator{\ord}{ord}

\begin{document}

\title[Homotopy ribbon concordance]
{Homotopy ribbon concordance, Blanchfield pairings,\\ and twisted Alexander polynomials}

\newcommand{\myemail}[1]{\href{mailto:#1}{#1}}

\author[S.~Friedl]{Stefan Friedl}
\address{Department of Mathematics, University of Regensburg, Germany}
\email{\myemail{sfriedl@gmail.com}}
\author[T.~Kitayama]{Takahiro Kitayama}
\address{Graduate School of Mathematical Sciences, The University of Tokyo, Japan}
\email{\myemail{kitayama@ms.u-tokyo.ac.jp}}
\author[L.~Lewark]{Lukas Lewark}
\address{Department of Mathematics, University of Regensburg, Germany}
\email{\myemail{lukas@lewark.de}}
\author[M.~Nagel]{\\Matthias Nagel}
\address{Department of Mathematics, ETH Zurich, Switzerland}
\email{\myemail{matthias.nagel@math.ethz.ch}}
\author[M.~Powell]{Mark Powell}
\address{Department of Mathematical Sciences, Durham University, United Kingdom}
\email{\myemail{mark.a.powell@durham.ac.uk}}

\hypersetup{pdfauthor={\authors},pdftitle={\shorttitle}}

\def\subjclassname{\textup{2020} Mathematics Subject Classification}
\expandafter\let\csname subjclassname@1991\endcsname=\subjclassname
\expandafter\let\csname subjclassname@2000\endcsname=\subjclassname

\begin{abstract}
We establish homotopy ribbon concordance obstructions coming from the Blanchfield form and Levine-Tristram signatures.
Then, as an application of twisted Alexander polynomials, we show that for every knot $K$ with nontrivial Alexander polynomial, there exists an infinite family of knots that are all concordant to $K$ and have the same Blanchfield form as~$K$, such that no pair of knots in that family is homotopy ribbon concordant.
\end{abstract}

\maketitle

\section{Introduction}

We study homotopy ribbon concordance of knots, extending the work~\cite{FP19} by the first and fifth authors on the classical Alexander polynomial to the Blanchfield form and to Levine-Tristram signatures.
Then we present an application of twisted Alexander polynomials to homotopy ribbon concordance.
As described in~\cite{FP19}, the classical Alexander polynomial is useful to show that there exists an infinite family of concordant knots that are not homotopy ribbon concordant to each other.
In this paper we show that there exists such an infinite family of knots even with isomorphic Seifert forms, and so also isomorphic Blanchfield forms,  using twisted Alexander polynomials.

Let $J$ and $K$ be oriented knots in $S^3$, and let $\pi_J$ and $\pi_K$ be the knot groups of $J$ and~$K$, respectively.
A locally flat oriented annulus $C$ properly embedded in $S^3 \times [0, 1]$, is called a \emph{homotopy ribbon concordance} from $J$ to $K$ if the following conditions are satisfied, where we write $\pi_C = \pi_1 (S^3 \times [0, 1] \setminus C)$:
\begin{enumerate}
\item $\partial C = -J \times \{ 0 \} \cup K \times \{ 1 \}$;
\item the induced homomorphism $\iota_J \colon \pi_J \to \pi_C$ is surjective;
\item the induced homomorphism $\iota_K \colon \pi_K \to \pi_C$ is injective.
\end{enumerate}
If there exists a homotopy ribbon concordance from $J$ to $K$, then we say that $J$ is \emph{homotopy ribbon concordant} to $K$, and write $J \geq_{\top} K$. 
In particular, in this case~$J$ and~$K$ are topologically concordant.
A knot $K$ is called a \emph{homotopy ribbon knot} if $K$ is homotopy ribbon concordant to the unknot.
As we discuss in more detail around \cref{question:transitivity} below, the relation $\geq_{\top}$ is conjectured to be a partial order on the set of isotopy classes of knots.

The notion of a homotopy ribbon concordance is a topological analogue of the notion of a (smooth) \emph{ribbon concordance}, defined by Gordon~\cite{Go81} as a smooth, oriented annulus $C$ properly embedded in $S^3\times[0,1]$ satisfying condition (1), such that after a small isotopy, the canonical projection of $C$ to $[0,1]$ is a Morse function without local minima.
A ribbon concordance from $J$ to $K$ satisfies conditions (1), (2) and (3), and thus is a homotopy ribbon concordance from $J$ to $K$ \cite[Lemma 3.1]{Go81}.
A knot $K$ is a \emph{ribbon knot} if and only if there is a ribbon concordance from $K$ to the unknot.
For further classical work on ribbon concordances we refer the reader to \cite{Gi84, Miy90, Miy98, Si92}.
Recently, a number of articles \cite{JMZ19, LZ19, MZ19, Sa19, Z19} have related the Heegaard Floer and the Khovanov homology theories to ribbon concordances of knots, providing evidence for Gordon's conjecture that ribbon concordance is a partial order. 
Those relationships rely on smooth techniques and do not generalize to homotopy ribbon concordances of knots.

On the other hand, we prove that the following statement, proven first by Gilmer~\cite{Gi84} for ribbon concordances using smooth methods, does in fact hold for homotopy ribbon concordances.
We denote the complement of an open tubular neighborhood of $K$ in $S^3$ by $X_K$.
Let us write $H_1(X_K;\Z[t^{\pm 1}])$ for the Alexander module of~$K$,
and $\Bl_K$ for the Blanchfield pairing of $K$, which is a sesquilinear Hermitian pairing on $H_1(X_K;\Z[t^{\pm 1}])$ with values in $\Q(t)/\Z[t^{\pm 1}]$.
For a submodule $G \subset H_1(X_K;\Z[t^{\pm 1}])$,
let $G^{\bot} \coloneqq \{a\in H_1(X_K;\Z[t^{\pm 1}])\mid  \Bl_K(a,b)=0 \text{ for all } b\in G\}$.
\begin{restatable}{thm}{thmblanchfield}\label{thm_blanchfield}
If $J \geq_{\top} K$, then
there exists a submodule $G\subset H_1(X_J; \Z[t^{\pm 1}])$
such that
$G\subset G^{\bot}$
and
the pairing on $G^{\bot} / G$ induced by $\Bl_J$
is isometric to
$\Bl_K$.
\end{restatable}
This theorem extends the result~\cite[Theorem 1.1]{FP19} that for $J\geq_{\top} K$, the Alexander polynomial of $K$ divides that of $J$, i.e.~$\Delta_K ~|~ \Delta_J$.  This should be contrasted with the Fox-Milnor condition~\cite{FM66} that for concordant knots $K$ and $J$, up to units we have $\Delta_K \cdot f \cdot \ol{f} = \Delta_J \cdot g \cdot \ol{g}$ for some $f,g \in \Z[t^{\pm 1}]$ with $|f(1)|=1=|g(1)|$. See \cref{remark-contrasting-Bl} for further discussion contrasting \cref{thm_blanchfield} with the analogous statements for concordance.

Gilmer showed that any ribbon concordance obstruction that is determined by the isometry class
of the Blanchfield pairing (which contains the same information as the S-equivalence class of the Seifert form \cite{T73}) is subsumed by the obstruction given in \cref{thm_blanchfield}. However, that obstruction is not easily testable.
So we derive homotopy ribbon obstructions coming from the homology of the double branched cover, and from Levine-Tristram signatures, in \cref{cor:double,cor:sig} respectively.
As far as we know, those obstructions are new even in the smooth category, i.e.~as ribbon concordance obstructions.

Then we move on to an application of twisted invariants. The following theorem shows the subtlety of the homotopy ribbon concordance relation, even for smoothly concordant knots with isomorphic Seifert forms.%
\begin{restatable}{thm}{thmmain}\label{thm_main}
Let $K$ be an oriented knot in $S^3$ with nontrivial Alexander polynomial $\Delta_K$, and let $V$ be a Seifert form of $K$. Then there exists an infinite family $\{ K_i \}_{i=1}^\infty$ of oriented knots in $S^3$ satisfying the following.
\begin{enumerate}[label=\emph{(\roman*)}]
\item For every $i$, $K_i$ has a Seifert form isomorphic to $V$.
\item For every $i$, $K_i$ is ribbon concordant to $K$.
\item For every $i \neq j$, $K_i$ is not homotopy ribbon concordant to $K_j$.
\end{enumerate}
\end{restatable}

Of course, the second item implies that $K_i$ is concordant to $K_j$ for every $i,j$. The third item implies~\cite[Lemma 3.1]{Go81} that $K_i$ is not ribbon concordant to $K_j$.
We believe that the weaker version of \cref{thm_main}, stated entirely in the smooth category by deleting `homotopy' from (iii), is also new.  Moreover, it seems to be hard to prove this statement using Heegaard Floer or Khovanov homology. The theorems from the articles \cite{JMZ19, LZ19, MZ19, Sa19, Z19} cited above show that a ribbon concordance from $J$ to $K$ induces a bigraded injection $\mathcal{H}(K) \hookrightarrow \mathcal{H}(J)$, for $\mathcal{H}$ a suitable knot homology theory.  This is extremely powerful for obstructing ribbon concordances for particular pairs of knots. However, general satellite formulae for these invariants are either very complicated, or in the case of Khovanov homology still under development.
Thus, with currently known computational methods, it seems hopeless to prove a result as general as \cref{thm_main}
by constructing the family $K_i$ as satellites, and then applying these injectivity results.

We prove \cref{thm_main} by applying twisted Alexander polynomials~\cite{Lin01, W94}.
In fact, we establish an obstruction to homotopy ribbon concordance in terms of these invariants.
Suppose that $J \geq_{\top} K$ with a homotopy ribbon concordance~$C$.
Let $\alpha \colon \pi_C \to \GL(n, R)$ be an $n$-dimensional representation over a Noetherian UFD~$R$.
We write $\Delta_J^{\alpha \circ \iota_J}(t)$ and $\Delta_K^{\alpha \circ \iota_K}(t)$ for the twisted Alexander polynomials of $J$ and $K$ associated to the induced representations $\alpha \circ \iota_J$ and $\alpha \circ \iota_K$ respectively.
The invariants are elements in $R[t^{\pm 1}]$, and well-defined up to multiplication by a unit in~$R[t^{\pm 1}]$.
We write $f ~|~ g$ for $f, g \in R[t^{\pm 1}]$ if $g = f h$ for some $h \in R[t^{\pm 1}]$. In particular,
$0 ~|~ g \Leftrightarrow 0 = g$.

The following theorem generalizes the result~\cite[Theorem 1.1]{FP19} on the classical Alexander polynomial to the case of twisted Alexander polynomials.
\begin{restatable}{thm}{thmdivision} \label{thm_division}
For a homotopy ribbon concordance $C$ from $J$ to $K$ and a representation $\alpha \colon \pi_C \to \GL(n, R)$, the following holds:
\[ \Delta_K^{\alpha \circ \iota_K} ~|~ \Delta_J^{\alpha \circ \iota_J}. \]
\end{restatable}

Again, we can contrast this with the Kirk-Livingston condition~\cite{KL99} on twisted Alexander polynomials of concordant knots, which is a direct analogue of the Fox-Milnor condition mentioned above. 

Let us briefly sketch the proof of \cref{thm_main} assuming \cref{thm_division}.
The infinite family $\{K_i\}_{i=1}^{\infty}$ of knots in the statement of the theorem is produced by a satellite construction.
Let $K$ be an oriented knot with nontrivial Alexander polynomial and let $V$ be a Seifert form associated to a Seifert surface~$F$ for~$K$.
We pick a simple closed curve $A$ in $S^3$, unknotted in $S^3$ and disjoint from~$F$, and think of~$K$ as a knot in the solid torus exterior of~$A$.
Then, fixing a prime $p$, we consider the satellite knots~$K_q$ of $T(p, q) \sharp - T(p, q)$ with pattern~$K$ for primes $q \neq p$, where $T(p, q)$ is the $(p, q)$-torus knot.
By the construction we can check that the family of knots $\{ K_q \}_{q \neq p}$, with $q$ prime, satisfies conditions (i) and (ii) of \cref{thm_main}.
Also, for appropriate $A$ and $p$ we can see condition (iii) as follows.
Suppose that $K_q \geq_{\top} K_{q'}$ for $q, q' \neq p$ with a homotopy ribbon concordance $C$. The nontriviality of the Alexander polynomial of $K$ enables us to find a metabelian representation $\alpha \colon \pi_C \to \GL(n, \Z)$ factoring through a $p$-group such that if $\Delta_{K_q}^{\alpha \circ \iota_{K_q}}(t) ~|~ \Delta_{K_{q'}}^{\alpha \circ \iota_{K_{q'}}}(t)$, then $q = q'$.  The condition that the representation $\alpha$ factors through a $p$-group is important in showing that both of these twisted polynomials are nonzero.
Thus (iii) follows from \cref{thm_division}.

Let us conclude the introduction with some open questions
highlighting the differences between  ribbon concordance and homotopy ribbon concordance.
The set of ribbon concordances is closed under composition. We do not know whether this holds for homotopy ribbon concordances, too.
\begin{q}\label{question:transitivity}
Is the homotopy ribbon concordance relation transitive, i.e.~does $I \geq_{\top} J$ and $J \geq_{\top} K$ for oriented knots $I$, $J$, $K$ imply $I \geq_{\top} K$?
\end{q}
The transitivity of the surjectivity condition (2) is easily verified, but the same cannot be said for the injectivity condition (3).
An affirmative answer would be a very interesting first step towards showing that $\geq_{\top}$ is a partial order.  One also needs antisymmetry, that $K \geq_{\top} J$ and $J \geq_{\top} K$ implies $J$ and $K$ are isotopic. 

Let us emphasize that \cref{thm_blanchfield,thm_division} actually hold without condition~(3). So, one might be tempted to simply strike condition (3) from the definition of homotopy ribbon concordances; \cref{thm_blanchfield,thm_division} would still hold. However, there is some indication that this would be a less natural definition than the one we have given.
Let $K_1$, $K_2$ be two knots with trivial Alexander polynomial. For such knots, there exist discs
embedded properly and locally flat in $D^4$ with
$\partial D_i = K_i$ and $\pi_1(D^4\setminus D_i) \cong \Z$ \cite[Theorem~11.7B]{FQ90}.
Tubing together $D_1$ and $D_2$ yields a topological concordance from $K_1$ to $K_2$
satisfying conditions (1) and (2), but (if $K_2$ is nontrivial) not (3).  So then without (3), $\geq_{\top}$ has no hope of being a partial order.

This observation also implies that, to answer the following question positively, one must indeed make use of condition~(3).
\begin{q}
Does $J \geq_{\top} K$ imply $g(J) \geq g(K)$?
\end{q}
The analogous statement for ribbon concordance was established by Zemke using knot Floer homology \cite{Z19}.

One may easily show using embedded Morse theory (see e.g.~\cite{BP16}) that for smoothly concordant knots $K, J$, there exists a third knot $L$ that is ribbon concordant to both $K$ and $J$. Whether the analogous statement holds for homotopy ribbon concordance is less clear.
\begin{q} \label{Q:join}
Given two topologically concordant knots $K, J$, does there exist a third knot $L$ such that
$L \geq_{\top} K$ and $L \geq_{\top} J$?
\end{q}
Note that the homotopy slice-ribbon conjecture (which states that all topologically slice knots are homotopy ribbon knots) would imply an affirmative answer to \cref{Q:join}, since one could pick $L$ as $K \# -J \# J$.

This paper is organized as follows.
\cref{sec_TAP} provides a brief review of twisted Alexander polynomials of knots.
\cref{sec_Blanchfield} contains the proof of \cref{thm_blanchfield,cor:double,cor:sig}.
In \cref{sec_TAM} we generalize the results in~\cite{FP19} to the case of twisted Alexander modules, and prove \cref{thm_division}.
In \cref{sec_satellite} we describe a satellite construction producing such an infinite family of knots as in \cref{thm_main}, and a satellite formula of twisted Alexander polynomials.
In \cref{sec_MAP} we introduce certain nonzero twisted Alexander polynomials associated to metabelian representations, which are useful for applications of \cref{thm_division}.
\cref{sec_proof} fleshes out the details of the proof of \cref{thm_main} that was sketched above.
\cref{sec:app} discusses the base ring change properties of Blanchfield pairings that are needed in \cref{sec_Blanchfield}.

\subsection*{Acknowledgments}
The authors thank Maciej Borodzik for inspiring ideas regarding applications of the Blanchfield obstruction,
and the anonymous referee for fruitful suggestions.
Some of the work on this article was undertaken at the MPIM Bonn during the workshop on 4-manifolds in September 2019,
and some motivation came from the workshop on low-dimensional topology in Regensburg one month later.
TK was supported by JSPS KAKENHI Grant Numbers JP18K13404, JP18KK0380.
LL was supported by the Emmy Noether Programme of the DFG.
MN gratefully acknowledges support by the SNSF Grant 181199.

\section{Twisted Alexander polynomials} \label{sec_TAP}

We begin with the definition of twisted Alexander polynomials of knots~\cite{Lin01, W94}.
We also refer the reader to the survey papers~\cite{DFL15, FV11} for more details on the invariants.
In this section, let $R$ be a Noetherian UFD with (possibly trivial) involution $\ol{\,\cdot\,} \colon R \to R$, and $Q(R)$ is its quotient field.

\subsection{Twisted homology and cohomology groups}

Let $X$ be a path connected space having universal cover $\widetilde{X}$ and let $Y$ be a subspace of $X$.
Let $\widetilde{Y}$ be the pullback of $Y$ by the covering map $\widetilde{X} \to X$.
Note that $\pi_1(X)$ acts on $\widetilde{X}$ on the left via deck transformations.
The singular chain complex $C_*(\widetilde{X}, \widetilde{Y})$ of $(\widetilde{X}, \widetilde{Y})$ is a left $\Z \pi_1(X)$-module.
We write $\ol{C}_*(\widetilde{X}, \widetilde{Y})$ when we think of $C_*(\widetilde{X}, \widetilde{Y})$ as a right $\Z \pi_1(X)$-module, using the involution of $\Z \pi_1(X)$ reversing elements of $\pi_1(X)$.

Let $\alpha \colon \pi_1 (X) \to \GL_n(R)$ be a representation.
For each nonnegative integer $i$ we define the \emph{$i$-th twisted homology group} $H_i^\alpha(X, Y; R^n)$ and the \emph{$i$-th twisted cohomology group} $H_\alpha^i(X, Y; R^n)$ \emph{of $(X, Y)$ associated to $\alpha$} as:
\begin{align*}
H_i^\alpha(X, Y; R^n) &= H_i(\ol{C}_*(\widetilde{X}, \widetilde{Y}) \otimes_{\Z \pi_1(X)} R^n), \\
H_\alpha^i(X, Y; R^n) &= H^i(\hom_{\Z \pi_1(X)}(C_*(\widetilde{X}, \widetilde{Y}), R^n)).
\end{align*}
When $Y$ is empty, we write $H_i^\alpha(X; R^n)$ and $H_\alpha^i(X; R^n)$ respectively.

If $(X, Y)$ is a CW-pair, then the cellular twisted homology and cohomology groups are similarly defined for the cellular chain complex of $(\widetilde{X}, \widetilde{Y})$, and isomorphic to the singular twisted homology and cohomology groups respectively.

The $0$-th twisted homology and cohomology groups are computed as follows (see for instance \cite[Proposition 3.1]{HS71}):
\begin{align*}
H_0^\alpha(X; R^n) &= R^n / \{ \alpha(\gamma) v - v \mid \gamma \in \pi_1(X), v \in R^n \}, \\
H_\alpha^0(X; R^n) &= \{ v \in R^n \mid \alpha(\gamma) v = v ~\text{for all $\gamma \in \pi_1(X)$} \}.
\end{align*}

\subsection{Twisted Alexander polynomials of knots}\label{subsec:twistedalexofknots}

Let $K$ be an oriented knot in~$S^3$.
Recall that we denote the complement of an open tubular neighborhood of $K$ in $S^3$ by $X_K$, and set $\pi_K = \pi_1(X_K)$.
Let $\phi_K \colon \pi_K \to \Z$ be the abelianization map sending a meridional element to $1$.

Let $\alpha \colon \pi_K \to \GL(n, R)$ be a representation.
We write $\alpha \otimes \phi_K \colon \pi_K \to \GL(n, R[t^{\pm 1}])$ for the tensor representation given by
\[ \alpha \otimes \phi_K(\gamma) = \alpha(\gamma) t^{\phi_K(\gamma)} \]
for $\gamma \in \pi_K$.
We call $H_1^{\alpha \otimes \phi_K}(X_K; R[t^{\pm 1}]^n)$ the \emph{twisted Alexander module of $K$ associated to $\alpha$}, which is a finitely generated $R[t^{\pm 1}]$-module.
We define the \emph{twisted Alexander polynomial} $\Delta_K^\alpha(t) \in R[t^{\pm 1}]$ of $K$ associated to $\alpha$ to be its order.
Namely, for an exact sequence
\[ R[t^{\pm 1}]^l \xrightarrow{r} R[t^{\pm 1}]^m \to H_1^{\alpha \otimes \phi_K}(X_K; R[t^{\pm 1}]^n) \to 0 \]
with $l \geq m$, $\Delta_K^\alpha(t)$ is the greatest common divisor of the $m$-minors of a representation matrix of $r$, and is well-defined up to multiplication by a unit in $R[t^{\pm 1}]$.
In the following we write $f \doteq g$ for $f, g \in R[t^{\pm 1}]$ if $f = u g$ for some unit $u \in R[t^{\pm 1}]$.

\begin{rem} \label{rem_TAP}~
\begin{enumerate}
\item The twisted Alexander polynomial of $K$ associated to the trivial representation $\pi_K \to \GL(1, \Z)$ coincides with the classical Alexander polynomial $\Delta_K(t)$ of $K$.
\item The twisted Alexander polynomial $\Delta_K^\alpha(t)$ is invariant under conjugation of representations~$\alpha$.
\end{enumerate}
\end{rem}

\section{The Blanchfield pairing} \label{sec_Blanchfield}

\subsection{Proof of the Blanchfield homotopy ribbon obstruction}
Recall that for a knot $K$, $X_K$ denotes the complement of an open tubular neighborhood of $K$ in $S^3$,
and the Alexander module of $K$ is $H_1(X_K; \Z[t^{\pm 1}])$, where $t$ acts as a generator of the deck transformation group of the infinite cyclic covering of $X_K$.
The \emph{Blanchfield pairing} $\Bl_K$ of $K$ is a nonsingular sesquilinear Hermitian form
\[
\Bl_K \colon H_1(X_K; \Z[t^{\pm 1}]) \times H_1(X_K; \Z[t^{\pm 1}]) \to \Q(t)/\Z[t^{\pm 1}],
\]
which may be defined by setting its adjoint $x\mapsto \Bl_K(-,x)$ to be the composition
\begin{multline}\label{eq:defBl}
	H_1(X_K; \Z[t^{\pm 1}]) \xrightarrow{\makebox[2.5em]{}} H_1(X_K, \partial X_K; \Z[t^{\pm 1}])
			 \xrightarrow{\makebox[2.5em]{\footnotesize PD$^{-1}$}} H^2(X_K; \Z[t^{\pm 1}])
			 \\ \xrightarrow{\makebox[2.5em]{$\scriptstyle\beta_K^{-1}$}} H^1(X_K; \Q(t)/\Z[t^{\pm 1}])
			 \xrightarrow{\makebox[2.5em]{$\kappa$}} \ohom_{\Z[t^{\pm 1}]}( H_1(X_K; \Z[t^{\pm 1}]), \Q(t)/\Z[t^{\pm 1}] ),
\end{multline}
where the first map is the natural one, PD denotes the Poincaré duality map,
$\beta_K$
denotes the Bockstein connecting homomorphism in cohomology induced by
the short exact sequence
$0 \to \Z[t^{\pm 1}] \to \Q(t) \to \Q(t)/\Z[t^{\pm 1}] \to 0$,
and $\kappa$ is the so-called Kronecker evaluation
(see e.g.~\cite{FP17} for details regarding this definition).
The involution $\ol{\,\cdot\,}$ on $\Z[t^{\pm 1}]$ is given by $\ol{p(t)} = p(t^{-1})$,
and for a $\Z[t^{\pm 1}]$-module $M$, $\ol{M}$ denotes the module given by the same
abelian group as $M$, and $p(t)$ acting as $\ol{p(t)}$.
We write $\ol{\hom(\,\cdot\,,\,\cdot\,)}$ as $\ohom(\,\cdot\,,\,\cdot\,)$.
A submodule $P \subset H_1(X_K; \Z[t^{\pm 1}])$ satisfying $P = P^\perp$ with respect to $\Bl_K$ is called a \emph{metabolizer}.

Let us now recall \cref{thm_blanchfield} from the introduction, and prove it.
\thmblanchfield*
\begin{proof}
Let $C$ be a homotopy ribbon concordance from $J$ to~$K$
and let $X_C$ be its exterior,
i.e.~the complement of an open tubular neighborhood of $C$ in $S^3\times[0,1]$.
See \cite{FNOP19} for a thorough treatment of tubular neighborhoods of submanifolds of a topological $4$-manifold.
The boundary $\partial X_C$ is homeomorphic to the gluing of $X_{-J}$ and $X_K$ along their boundary tori,
gluing meridian to meridian and longitude to longitude.

The composition $j\colon X_{J} \to X_C$ of orientation reversal and inclusion
induces a surjection $j_*\colon H_1(X_J;\Z[t^{\pm 1}]) \to H_1(X_C;\Z[t^{\pm 1}])$,
while the inclusion
$k\colon X_K \to X_C$ induces
an injection $k_*\colon H_1(X_K;\Z[t^{\pm 1}]) \to H_1(X_C;\Z[t^{\pm 1}])$.
This was shown in~\cite{FP19}; the proof (also for twisted homology) can be found in \cref{sec_TAM}.

The surjectivity of $j_*$ implies that the kernel of the homomorphism induced by the inclusion
$\partial X_C\to X_C$ is a metabolizer of $-\Bl_J \oplus \Bl_K$ (see \cite{FLNP17} for additivity of the Blanchfield form and~\cite[Prop.~8.2]{F04} for the proof that the kernel is a metabolizer).
The remainder of the proof is an algebraic argument that we outsource to the following lemma,
applied with $e = j_*$, $m = k_*$, $\lambda_E = \Bl_J$, $\lambda_M = \Bl_K$, and $R = \Z[t^{\pm 1}]$.
Such a submodule $G \subset H_1(X_J; \Z[t^{\pm 1}])$ as in the theorem can be chosen to be $\ker (j_*)$.
\end{proof}

For a commutative ring $R$, we write $Q(R)$ for the total quotient ring of $R$, i.e.\ the localization of $R$ with respect to the multiplicative set of elements of $R$ that are not zero divisors.
An element $a$ of an $R$-module is called \emph{torsion} if it is annihilated by some $r\in R$ that is not a zero divisor,
and an $R$-module is called \emph{torsion} if all its elements are torsion.
For a finitely generated torsion $R$-torsion module $M$ with a nonsingular sesquilinear form $\lambda \colon M \times M \to Q(R) / R$, a submodule $P \subset M$ satisfying $P = P^\perp$ with respect to $\lambda$ is called a \emph{metabolizer}.
\begin{lem}
Let $R$ be a commutative ring with an involution.
Let $M, E$ be finitely generated $R$-torsion modules with sesquilinear, Hermitian and nonsingular pairings
$\lambda_M \colon M\times M\to Q(R)/R$ and
$\lambda_E \colon E\times E\to Q(R)/R$.
Let $N$ be an $R$-module with an epimorphism $e\colon E\to N$ and a monomorphism $m\colon M\to N$, such
that the kernel of $(e + m)\colon E\oplus M \to N$ is a metabolizer for~$-\lambda_E \oplus \lambda_M$.
Let $G$ be the kernel of $e$. Then
the pairing induced by $\lambda_E$ on $G^{\bot} / G$ is isometric to $\lambda_M$.
\end{lem}

\begin{proof}
Let $S = e^{-1}(\im(m)) \subset E$ and define $g\colon S\to M$ as $m^{-1}\circ e$.
Let us check that $g$ respects the pairings. Let $a_1, a_2 \in S$ be given. Then for $i\in \{1,2\}$, we have
$e(a_i) = m(g(a_i))$, so $a_i - g(a_i) \in \ker(e + m)$. Since $\ker(e + m)$ is a metabolizer for $-\lambda_E \oplus \lambda_M$, we have
\[
-\lambda_E(a_1, a_2) + \lambda_M(-g(a_1), -g(a_2)) = 0 ~\Rightarrow~
\lambda_E(a_1, a_2) = \lambda_M(g(a_1), g(a_2))
\]
as claimed.
Next, $g$ is clearly surjective and has kernel $G = \ker e$, so it induces an isometry $S/G \to M$.
Let us now prove that $S = G^{\bot}$.

To show $S\subset G^{\bot}$, let $a\in S$. If $b\in G$, then
\[
\lambda_E(a,b) = -\bigl(-\lambda_E(a,b) + \lambda_M(-g(a),0)\bigr),
\]
which equals zero since $a - g(a)$ and $b$ are both contained in the metabolizer $\ker(e + m)$.

To show $G^{\bot} \subset S$, let $c\in G^{\bot}$ be given. Consider the homomorphism $S/G \to Q(R)/R$
given by $[a] \mapsto \lambda_E(a,c)$. The pairing induced on $S/G$ by $\lambda_E$ is nonsingular since
it is isometric to $\lambda_M$.
So its adjoint map is an isomorphism, whence there exists $[c'] \in S/G$
such that for all $[a] \in S/G$, we have $\lambda_E(a,c) = \lambda_E(a,c')$.
Thus for all $a\in S$, we have $\lambda_E(a, c-c') = 0$. If $a\in E$ and $b\in M$ with $e(a) + m(b) = 0$
are given, then $a\in S$, and so $-\lambda_E(a, c-c') + \lambda_M(b,0) = 0$.
Thus $c-c' \in \ker(e + m)^{\bot} = \ker(e + m)$, and thus $c - c' \in \ker(e + m) \cap E = G \subset S$.
Since $c'$ is also in~$S$, it follows that $c\in S$.
\end{proof}

\subsection{Homotopy ribbon obstructions from the double branched covering and signatures}

Let us write $[K]$ for the S-equivalence class of a knot $K$, i.e.\ $[K] = [K']$ if and only if
$\Bl_K$ and $\Bl_{K'}$ are isometric.
Write \[[J] \geq_S [K]\] if the conclusion of the previous theorem holds, i.e.~if there exists a submodule $G \subset H_1(X_J;\Z[t^{\pm 1}])$ such that $G\subset G^{\bot}$ and the pairing on $G^{\bot} / G$ induced by $\Bl_J$ is isometric to $\Bl_K$.
This relation was introduced by Gilmer \cite{Gi84}, who showed that if $J$ is ribbon concordant to $K$, then $[J]\geq_S [K]$.
\cref{thm_blanchfield} strengthens this statement, and may now be formulated as
\[
J\geq_{\top} K \Rightarrow [J]\geq_S[K].
\]
Gilmer provides an equivalent characterization of the relation $\geq_S$ in terms of Seifert matrices,
which constitutes his main technical tool.
Among other results, he shows that the relation $\geq_S$ is a partial order.
In this text, we will rely on the definition of $\geq_S$ via the Blanchfield pairing, instead of Seifert matrices.

\begin{rem}\label{remark-contrasting-Bl}
If $J$ and $K$ are topologically concordant knots, then $[J]$ and $[K]$ are \emph{algebraically concordant},
i.e.~$-\Bl_J \oplus \Bl_K$ admits a metabolizer.
Let us point out that $[J] \geq_S [K]$ implies that $[J]$ and $[K]$ are algebraically concordant
(and thus, as one might expect, our Blanchfield obstruction for homotopy ribbon concordance is indeed stronger than the usual Blanchfield obstruction for topological concordance).
Indeed, let $G \subset H_1(X_J;\Z[t^{\pm 1}])$ be chosen as above, such that
the pairing induced on $G^{\bot}/G$ by $\Bl_J$ is isometric to $\Bl_K$.
Let $g\colon G^{\bot} \to H_1(X_K; \Z[t^{\pm 1}])$ be the composition
of the projection to $G^{\bot}/G$ and the isometry to $\Bl_K$.
Then, one may check that a metabolizer of $-\Bl_J \oplus \Bl_K$ is given by
$
\{ (x, g(x)) \mid x \in G^{\bot} \} \subset H_1(X_J; \Z[t^{\pm 1}]) \oplus H_1(X_K; \Z[t^{\pm 1}]).
$
\end{rem}

There seems to be no known algorithm to decide whether $[J]\geq_S[K]$ holds for given knots $J$ and $K$.
On the level of modules (forgetting about the Blanchfield pairing),
the relation~$[J]\geq_S[K]$
implies that the Alexander module of $K$ is a quotient of a submodule of the Alexander module of~$J$.
But even this condition seems difficult to check for general given knots $J$ and $K$.

Switching to PID coefficients at least gives testable obstructions beyond divisibility of the Alexander polynomials.
In particular, in this way one obtains homotopy ribbon obstructions from the homology of the double branched covering,
and from Levine-Tristram signatures. Let us state these obstructions now.
Denote the double branched cover of $S^3$ along $K$ by $\Sigma_{K,2}$.
\begin{prop}\label{cor:double}
If $J$, $K$ are knots with $[J] \geq_S [K]$,
then $H_1(\Sigma_{K,2}; \Z)$ is isomorphic to a subgroup of $H_1(\Sigma_{J,2}; \Z)$.
In fact, there exist abelian groups $W, G$ and two short exact sequences%
\begin{equation}\label{eq:AA}\begin{tikzcd}[row sep=small]
&&& 0 \ar[d] \\
&&& G \ar[d] \\
0 \ar[r] & H_1(\Sigma_{K,2}; \Z) \ar[r] & H_1(\Sigma_{J,2}; \Z) \ar[r] & W \ar[r] \ar[d] & 0 \\
&&& G \ar[d] \\
&&& 0
\end{tikzcd}\end{equation}
\end{prop}
The proof will be given in \cref{section:Blanchfield-PID}.
Recall that for a knot~$K$, the order of $H_1(\Sigma_{K,2};\Z)$ is $\det K$.
So the existence of the abelian groups $W$ and $G$ satisfying \cref{eq:AA}
implies that $\det K$ divides $\det J$, and that the quotient $\det J / \det K$ is a square.
In fact, that much would already follow from $\Delta_K$ dividing $\Delta_J$ and the Fox-Milnor condition.
However, \cref{eq:AA} has stronger implications.
For example, if $[J] \geq_S [K]$ and $p$ is an odd prime not dividing $\det J / \det K$,
then the $p$-primary parts of $H_1(X_J;\Z)$ and $H_1(X_K;\Z)$ are isomorphic.
To give another example,
if for some odd prime $p$, the $p$-primary parts of
$H_1(X_J; \Z)$ and $H_1(X_K; \Z)$ are $\Z/p \oplus \Z/p^5$ and $\Z/p^2$, respectively, then
$[J] \not\geq_S [K]$ follows (because the cokernel $W$ of any injection $\Z/p^2 \to \Z/p \oplus \Z/p^5$
is isomorphic to $\Z/p \oplus \Z/p^3$, and this group $W$ does not admit a group $G$ fitting into \cref{eq:AA}).
In contrast, if $J$ and $K$ are merely assumed to be topologically concordant,
then the only statement that needs to follow about the groups $H_1(\Sigma_{J,2};\Z)$ and $H_1(\Sigma_{K,2};\Z)$
is that their direct sum has square order, i.e.~that $\det J\cdot\det K$ is a square.

Let us denote the Levine-Tristram signature and nullity, of a knot $K$ at $e^{\pi i x} \in S^1$ for $x\in \R$, by $\sigma_x(K)$ and $\eta_x(K)$ respectively.
We write $\deg_x(K)$ for the multiplicity of the root $e^{\pi ix}$ of $\Delta_K$, which equals $0$ if $e^{\pi ix}$ is not a root of $\Delta_K$.
In this notation, the classical knot signature corresponds to~$\sigma_1$.
Note that $\deg_x(K) \geq \eta_x(K)$ holds for all knots $K$ and all $x\in \R$.
\begin{prop}\label{cor:sig}
If $J$, $K$ are knots with $[J] \geq_S [K]$, then for all $x\in\R$,
\[
\deg_x(J) - \deg_x(K)
~\geq~
\eta_x(J) - \eta_x(K)
~\geq~
|\sigma_x(J) - \sigma_x(K)|.
\]
\end{prop}
The proof will be given in \cref{section:devissage}.
Note that in the particular case, where $e^{\pi i x}$ is not a root of the quotient $\Delta_J / \Delta_K$, one obtains that $\sigma_x(J) = \sigma_x(K)$.
In contrast, if $J$ and $K$ are merely assumed to be topologically concordant, only the following inequalities (which are weaker than the inequalities given by \cref{cor:sig}) need to hold for all $x\in\R$:
\[
\deg_x(J) + \deg_x(K)
~\geq~
\eta_x(J) + \eta_x(K)
~\geq~
|\sigma_x(J) - \sigma_x(K)|.
\]

The proofs of \cref{cor:double,cor:sig} follow a similar path for the start of their proofs.
The homology of the double branched covering and the Levine-Tristram signatures
are determined by the Blanchfield pairings $\Bl^R$ with certain PIDs $R$ as coefficient module.
As discussed in \cref{sec:app}, $\Bl$ determines $\Bl^R$,
and $[J] \geq_S [K]$ implies the condition for $\Bl^R$ analogous to \cref{thm_blanchfield},
which in turn (using that $R$ is a PID) implies \cref{eq:AA} in \cref{cor:double}
and the inequalities in \cref{cor:sig}.
Before delving into those proofs in full detail in \cref{section:Blanchfield-PID,section:devissage},
let us give two example applications.

\begin{exmp}
Let $J'$ and $K'$ be the knots $8_{18}$ and $8_{20}$ from the knot table \cite{LM20}.
Let $J = J'\# -J'$ and $K = K'\# -K'$. Clearly, $J$ and $K$ are ribbon knots.
We claim \cref{cor:double} implies that $J$ is not homotopy ribbon concordant to $K$,
whereas no previously available obstructions could even show that $J$ is not ribbon concordant to $K$.

Indeed, one computes 
$H_1(\Sigma_{J,2}; \Z) \cong H_1(\Sigma_{J',2}; \Z)^{\oplus 2} \cong (\Z/3)^{\oplus 4} \oplus (\Z/5)^{\oplus 2}$
and $H_1(\Sigma_{K',2};\Z) \cong H_1(\Sigma_{K,2};\Z)^{\oplus 2} \cong (\Z/9)^{\oplus 2}$.
Since $p = 3$ does not divide $\det J / \det K = 25$, and the $p$-primary parts of the groups
$H_1(\Sigma_{J,2}; \Z)$ and $H_1(\Sigma_{K',2};\Z)$ are not isomorphic, \cref{cor:double} indeed
implies that $J \not \geq_{\top} K$.

Let us now check previously known ribbon concordance obstructions.
The Alexander polynomial $\Delta_{J'} = (t^2 - t + 1)^2(t^2 - 3t + 1)$
is divisible by the Alexander polynomial $\Delta_{K'} = (t^2 - t + 1)^2$,
and so $\Delta_J = \Delta_{J'}^2$ is divisible by $\Delta_K = \Delta_{K'}^2$.
Thus the known Alexander polynomial homotopy ribbon obstruction \cite{FP19} is not applicable for $J$ and $K$.
Moreover, one may check by computer calculation that there exist bigraded injections
$\mathcal{H}(K) \hookrightarrow \mathcal{H}(J)$ for $\mathcal{H}$ either Khovanov homology,
or knot Floer homology. So the ribbon concordance obstructions of Khovanov homology \cite{LZ19}
and knot Floer homology \cite{Z19} are not applicable for $J$ and $K$, either.
In fact, one may avoid computer calculation and read the Khovanov homologies of $J$ and $K$
from their Jones polynomials and knot signatures, and the knot Floer homologies of $J$ and $K$ from
their Alexander polynomials and knot signatures, since both $J$ and $K$ are quasi-alternating \cite{MO08}.
\end{exmp}

\begin{exmp}
Let $K$ be the knot $12n_{582}$ from the knot table \cite{LM20}.
This knot is a ribbon knot with Alexander polynomial $(t^{-1}-1+t)^2$ and
nonconstant signature profile: $\deg_{1/3}(K) = 2, \eta_{1/3}(K) = 1, \sigma_{1/3}(K) = 1$.
As an example application of the Blanchfield obstructions, let us prove that if $J$ is a prime knot with crossing number~12 or less such that $J\geq_{\top} K$, then $J \in \{K, K'\}$, where $K'$ is the knot $10_{99}$ in the knot table.

First off, $J$ needs to be topologically slice. This leaves (up to symmetry, i.e.~counting mirror images and reverses only once) 159 knots from the table.
Among those, there are 134 knots with $\deg_{1/3} = 0$ (i.e.~$\Delta(e^{\pi i/3}) \neq 0$) and 10 knots with
${\deg_{1/3} = 2}$, ${\eta_{1/3} = 2}$, ${\sigma_{1/3} = 0}$.
These are ruled out by the first inequality in \cref{cor:sig}, whereas $-K$
is ruled out by the second inequality, since ${\deg_{1/3}(-K) = 2}$, ${\eta_{1/3}(-K) = 1}$, ${\sigma_{1/3}(-K) = -1}$.
There remain 15 knots. Among them, 13 have double branched covering with first homology group isomorphic to
$\Z/9, \Z/81$ or $\Z/9\oplus \Z/25$. So they are ruled out by \cref{cor:double}, since
the first homology group of the double branched covering of $K$ is $\Z/3\oplus \Z/3$, which does not inject into a cyclic group (note that here we are not using the vertical exact sequence in \cref{eq:AA}).

The two remaining knots are $K$ itself and $K'$ (the latter being amphicheiral). We have $H_1(K'; \Z) \cong \Z/9 \oplus \Z/9$ and $\deg_{1/3}(K') = 4, \eta_{1/3}(K') = 2, \sigma_{1/3}(K') = 0$, so \cref{cor:double,cor:sig} do not obstruct $K' \geq_{\top} K$. One may check that $K'$ is not ribbon concordant to~$K$, e.g.~using Khovanov homology~\cite{LZ19}.
However, we do not know the answer to the following.
\end{exmp}
\begin{q}
Does $[10_{99}] \geq_S [12n_{582}]$ hold? If so,
is $10_{99}$ homotopy ribbon concordant to $12n_{582}$?
\end{q}
\begin{rem}
In light of the obstructions for homotopy ribbon concordances coming from the Blanchfield pairing (\cref{thm_blanchfield}) and from twisted Alexander polynomials (\cref{thm_division}),
one might hope for an obstruction coming from twisted Blanchfield pairings.
We have not pursued this any further
in order to avoid overly complicating this text, and because it was not necessary to obtain the application in \cref{thm_main}.
\end{rem}

\begin{rem}
The obstruction given in \cref{cor:sig} could be strengthened a bit by considering the full isometry type of the Blanchfield pairing over $\mathbb{R}$, which is not completely captured by nullities, signatures and the Alexander polynomial.  Likewise, taking the linking pairing of the double branched cover into account, and not just its homology, would strengthen \cref{cor:double}.
\end{rem}

\subsection{The Blanchfield pairing with PID coefficients}\label{section:Blanchfield-PID}
This subsection is devoted to the proof of \cref{cor:double}. Let us jump directly into it, and prove the
crucial \cref{lem:pid} afterwards.

\begin{proof}[Proof of \cref{cor:double}]
First we appeal to two technical results dealing with change of base ring for the Blanchfield form, whose proofs we have relegated to \cref{sec:app}.
The $\Z[t^{\pm 1}]$-algebra $R\coloneqq \Z[t^{\pm 1}]/(t + 1)$
has only $(t+1)$-torsion, and is thus admissible in the terminology of \cref{sec:app}.
So, by \cref{prop:appmain},
for all knots $L$, there is a Blanchfield pairing $\Bl^R_L$ defined on $H_1(X_L; R)$,
which is determined by $\Bl_L$.
Moreover, by \cref{prop:appmain2}, $[J] \geq_S [K]$ implies that
there exists an $R$-submodule $G \subset H_1(X_J; R)$ such that $G\subset G^\bot$,
and $G^{\bot} / G$ is isomorphic to $H_1(X_K;R)$.

Note that $R \to \Z, t \mapsto -1$ is an isomorphism
of $\Z[t^{\pm 1}]$-algebras (with trivial involution), where $t$ acts on $\Z$ by $-1$.
In particular, as a ring $\Z[t^{\pm 1}]/(t + 1)$ is a PID. By \cref{lem:pid}~(iii) below,
there are two short exact sequences
\[
\begin{array}{c}
\begin{tikzcd}
0 \ar[r] & H_1(X_K;R) \ar[r] & H_1(X_J;R) \ar[r] & W \ar[r] & 0,
\end{tikzcd}
\\
\begin{tikzcd}
0 \ar[r] & G \ar[r] & W \ar[r] & G \ar[r] & 0.
\end{tikzcd}
\end{array}
\]
Note that for all knots~$L$, $H_1(\Sigma_{2,L}; \Z)$ and $H_1(X_L; R)$ are isomorphic abelian groups. Indeed, since $H_1(X_L;R) \cong H_1(X_L;\Z[t^{\pm 1}]) \otimes_{\Z[t^{\pm 1}]} R \cong H_1(X_L;\Z[t^{\pm 1}])/(t+1)$ by a straightforward application of the universal coefficient theorem (noting $\displaystyle{\Tor_1^{\Z[t^{\pm 1}]}(\Z,R) =0}$), both modules are presented by the matrix~$(tV-V^T)|_{t=-1} = -V-V^T$, where $V$ is a Seifert matrix for~$L$. Thus, we have indeed constructed the two short exact sequences \cref{eq:AA}.
\end{proof}
As a side remark, note that under the identification of
$H_1(X_L; R)$ with $H_1(\Sigma_{2,L}; \Z)$, we have $\Bl^R_L = 2\cdot \text{lk}_L$, where $\text{lk}_L$ denotes the usual linking pairing on $H_1(\Sigma_{2,L}; \Z)$.

Also, note that one could prove the conclusion of \cref{cor:double} under the stronger hypothesis that $J \geq_{\top} K$ instead of $[J] \geq_S [K]$, namely
by using that \cref{thm_blanchfield} holds not just for $\Z[t^{\pm 1}]$, but for all $\Z[t^{\pm 1}]$-algebras $R$ without $t-1$ and $\Xi$-torsion (cf.~\cref{sec:app}), with the analogous proof. This proof would not require the relationship between $\Bl$ and $\Bl^R$, but would result in a weaker version of \cref{cor:double}.

Now, let us come to \cref{lem:pid}. Let us consider modules and pairings over a general PID $R$ (PIDs are understood to be commutative in this text).
The \emph{order} $\ord M$ of a finitely generated torsion module $M$ over $R$ is defined as follows:
pick a decomposition of $M$ as sum of cyclic modules $R/(a_1) \oplus \ldots \oplus R/(a_n)$.
Then $\ord M \coloneqq a_1 \cdot \ldots \cdot a_n \in R$. The order is well-defined up to multiplication with units in $R$; we use the symbol $\doteq$ to denote that two elements of $R$ are equal up to multiplication with a unit.
Note that if $R$ has an involution $\ol{\,\cdot\,}$, then $\ord \ol{M} = \ol{\ord M}$.
This definition of order agrees with the more general definition in \cref{subsec:twistedalexofknots}
of orders of modules over rings that need not be PIDs.
\begin{lem}\label{lem:pid}
Let $R$ be a PID with an involution $\ol{\,\cdot\,}$.
\begin{enumerate}[label=\emph{(\roman*)}]
\item
Every finitely generated $R$-torsion module $M$ is $($noncanonically$)$
 isomorphic to $\hom(M, Q(R)/R)$.
\end{enumerate}
Let $A$ be a finitely generated $R$-torsion module and let $\lambda\colon A\times A\to Q(R)/R$ be a sesquilinear Hermitian nonsingular pairing.
\begin{enumerate}[resume,label=\emph{(\roman*)}]
\item For $M$ a submodule of $A$, we have $\ord M \cdot \ord M^{\bot} \doteq \ord A$ and $M = M^{\bot\bot}$.
\item For a submodule $G \subset A$ such that $G\subset G^{\bot}$, there is an $R$-module $W$ fitting into the following two $($nonnatural$)$ short exact sequences as follows.
\[
\begin{tikzcd}[row sep=small]
&&& 0 \ar[d] \\
&&& \ol{G} \ar[d] \\
0 \ar[r] & {G^{\bot} / G} \ar[r] & A \ar[r] & W \ar[r] \ar[d] & 0. \\
&&& {G} \ar[d] \\
&&& 0
\end{tikzcd}
\]
\end{enumerate}
\end{lem}

\begin{proof}
(i)
Cyclic torsion modules are (noncanonically) isomorphic to $R / (a)$ for some $a\in R\setminus\{0\}$, and an isomorphism $R/(a) \to \hom(R/(a), Q(R)/R)$ is given by $1\mapsto (1\mapsto 1/a)$.
This suffices, since $M$ decomposes as a sum of cyclic torsion modules, and $\hom(N\oplus N', Q(R)/R) \cong
\hom(N, Q(R)/R)\oplus \hom(N', Q(R)/R)$.
\smallskip

(ii)
There is a short exact sequence
\begin{equation}\label{eq:dualquotient}
\begin{tikzcd}
0 \ar[r] & M^{\bot} \ar[r] & A \ar[r] & \ohom(M,Q(R)/R) \ar[r] & 0,
\end{tikzcd}
\end{equation}
where the second map is the inclusion, and the third map is the composition of the adjoint map $A\to \ohom(A,Q(R)/R)$ of $\lambda$ (which is an isomorphism) and the
map $\ohom(A,Q(R)/R) \to \ohom(M,Q(R)/R)$, which is surjective
since $Q(R)/R$ is an injective $R$-module, and so the contravariant functor $\ohom(-,Q(R)/R)$ is exact.
The orders of modules in a short exact sequence are multiplicative,
so $\ord M^{\bot} \cdot \ord \ohom(M,Q(R)/R) \doteq \ord A$,
which implies $\ord M^{\bot} \cdot \ord \ol{M} \doteq \ord A$ using (i).
 In the same way, we obtain $\ord M^{\bot\bot} \cdot \ord \ol{M^{\bot}} \doteq \ord A$. Thus we get
$\ord M \doteq \ord M^{\bot\bot}$.
The inclusion $M\subset M^{\bot\bot}$ follows directly from the definition of $^\bot$, and since $M$ and $M^{\bot\bot}$ have the same order, they are equal.
\smallskip

(iii)
Let us write $\iota$ for inclusions of submodules of $A$ into $A$, and
$\psi\colon A\to \ohom(A,Q(R)/R)$ for the adjoint of $\lambda$.
Let $\beta$ be the composition of the maps
\[
\begin{tikzcd}
G^{\bot}  \arrow[r,"\iota",hook] &
A \arrow[r,"\psi"] &
\ohom(A,Q(R)/R) \arrow[r,"\iota^*", two heads] &
\ohom(G^{\bot},Q(R)/R).
\end{tikzcd}
\]

The kernel of $\beta$ is $G^{\bot\bot} = G$.
So $\beta$ induces an injection
\[
\tilde\beta\colon {G^{\bot} / G} \to \ohom(G^{\bot},Q(R)/R).
\]
Fix an isomorphism $\alpha\colon {\ohom(G^{\bot},Q(R)/R)}\to \ol{G^{\bot}}$.
We get the following commutative diagram, in which the two rows and the first column are exact;
this induces the dashed maps, such that the second column is also exact.
\[
\begin{tikzcd}
 && 0 \ar[d] &0\ar[d] \\
0 \ar[r] & {G^{\bot} / G} \ar[r,"\alpha\circ\tilde\beta"]\ar[d,"\text{id}"] & \ol{G^{\bot}} \ar[r]\ar[d,"\iota"] & \ol{G^{\bot}}/\im\alpha\tilde\beta \ar[r]\ar[d,dashed] & 0 \\
  0 \ar[r] & {G^{\bot} / G} \ar[r,"\iota\circ\alpha\circ\tilde\beta"] & \ol{A} \ar[r]\ar[d] & \ol{A}/\im \iota\alpha\tilde\beta \ar[r]\ar[d,dashed] & 0 \\
&& \ol{A/G^{\bot}} \ar[d]\ar[r,"\text{id}"] & \ol{A/G^{\bot}}\ar[d] \\
&& 0 & 0
\end{tikzcd}
\]
The second row and second column of this diagram form the desired diagram, using that $\ol{A}\cong A$
and setting $W = \ol{A}/\im \iota\alpha\tilde\beta$.
It just remains to check that
$\ol{A/G^{\bot}}\cong G$ and $\ol{G^{\bot}} / \im \alpha\tilde\beta \cong \ol{G}$.

An isomorphism $\ol{A/G^{\bot}} \cong \hom(G,Q(R)/R)$ is given by considering the short exact sequence~\cref{eq:dualquotient},
and $\hom(G,Q(R)/R)$ is isomorphic to ${G}$ by~(i).

In the following commutative diagram, the rows are exact (the surjectivity of the second map in the top row follows from the injectivity of the module $Q(R)/R$ as in the proof of (ii) above), and the first two vertical maps are isomorphisms.
\[
\begin{tikzcd}
0 \ar[r] & {G^{\bot} / G} \ar[r,"\tilde\beta"]\ar[d,"\text{id}"] & \ohom(G^{\bot},Q(R)/R) \ar[d,"\alpha"]\ar[r] & \ohom(G,Q(R)/R) \ar[r]\ar[d,dashed] & 0 \\
0 \ar[r] & {G^{\bot} / G} \ar[r,"\alpha\circ\tilde\beta"] & \ol{G^{\bot}} \ar[r] & \ol{G^{\bot}}/\im\alpha\tilde\beta \ar[r] & 0
\end{tikzcd}
\]
This induces the desired isomorphism $G \cong \ohom(G,Q(R)/R) \to \ol{G^{\bot}} / \im \alpha\tilde\beta$, drawn dashed.%
\end{proof}
As an illustration of how the previous lemma can be used,
we give a new proof of the following corollary, by applying the lemma to $R = \Q[t^{\pm 1}]$.
The corollary was originally shown by Gilmer using Seifert matrices.

\begin{cor}[\cite{Gi84}]\label{cor:Bl-isometric}
If $J$ and $K$ are knots with $[J] \geq_S [K]$ and $\Delta_J \doteq \Delta_K$, then $\Bl_J$ and $\Bl_K$ are isometric.
\end{cor}

\begin{proof}
By definition of $\geq_S$, there exists a submodule $G\subset H_1(X_J; \Z[t^{\pm 1}])$ such that $G\subset G^{\bot}$
and the pairing induced by $\Bl_J$ on $G^{\bot} / G$ is isometric to $\Bl_K$.
Let $\widetilde G = G \otimes \Q$. By \cref{prop:appmain2}, $\widetilde G^{\bot} / \widetilde G$ is isomorphic to
$H_1(X_K; \Q[t^{\pm 1}])$, so
\[
\ord \widetilde G \cdot \Delta_K  \doteq \ord\widetilde G^{\bot}.
\]
Moreover, by \cref{lem:pid}~(ii), we have
\[
\ord \widetilde G \cdot \ord \widetilde G^{\bot} \doteq \Delta_J.
\]
Since $\Delta_K \doteq \Delta_J$ by assumption, those two equations imply that $\ord \widetilde G \doteq 1$, which means  that $\widetilde G$ is trivial. This is equivalent to $G$ being $\Z$-torsion. However, $H_1(X_J; \Z[t^{\pm 1}])$ is $\Z$-torsion free (see \cref{lem:alexztorsionfree}), and thus its submodule $G$ is as well. Hence $G$ is trivial, and the pairing induced by $\Bl_J$ on $G^{\bot} / G \cong H_1(X_J; \Z[t^{\pm 1}])$, to which $\Bl_K$ is isometric, is just $\Bl_J$ itself.
\end{proof}

\subsection{Dévissage of linking pairings over a PID}\label{section:devissage}
This subsection is devoted to the proof of \cref{cor:sig}.
As will be explained in detail later, $\deg_x, \sigma_x$ and $\eta_x$
can be read from the Blanchfield pairing $\Bl^{\R[t^{\pm 1}]}$ with $\R[t^{\pm 1}]$ coefficients,
which is in turn determined by $\Bl$, as discussed in \cref{sec:app}.
Since $\R[t^{\pm 1}]$ is a PID, the results of \cref{lem:pid} apply.
But this is not enough; to prove \cref{cor:sig},
we need a clearer understanding of linking pairings over $\R[t^{\pm 1}]$.
Localizing this ring will simplify things.
Let us work in a more general setting. Let $R$ be a discrete valuation ring, or DVR for short (i.e.~a commutative PID with a unique nonzero maximal ideal), equipped with an involution $\ol{\,\cdot\,}$, satisfying the two following conditions:
\begin{gather}\label{eq:technical1}
\text{There is a generator $\tau\in R$ of the unique maximal ideal of $R$ with $\tau=\ol{\tau}$}. \\
\label{eq:technical2}
\text{There is a unit $s\in R$ such that $s + \ol{s} = 1$}.
\end{gather}
Here is the key example of such a ring.

\begin{lem}\label{lem:DVR-example}
  Let $\zeta \in \R[t^{\pm 1}]$ be an irreducible polynomial with $\zeta = \ol{\zeta}$. The localization $R:= \R[t^{\pm 1}]_{(\zeta)}$ of $\R[t^{\pm 1}]$ at the ideal $(\zeta)$ is a DVR satisfying \eqref{eq:technical1} and \eqref{eq:technical2} with $\tau = \zeta$ and $s=1/2$.
\end{lem}
\begin{proof}
Since $\zeta$ is irreducible, the ideal $(\zeta)$ is prime.
Now, $R$ is the localization of a PID at a prime ideal, and as such $R$ is a DVR.
\end{proof}

Let us give some motivation for the conditions \cref{eq:technical1,eq:technical2}.
Symmetric bilinear forms are usually dramatically simplified by imposing the condition that the
ground field has $\text{char} \neq 2$ (or that 2 be invertible in the ground ring).
Condition \cref{eq:technical2} is simply the analogue condition for Hermitian sesquilinear forms.
It is needed below in the proof of \cref{lem:localpid}(iv), and later on to simplify
Hermitian pairings over the field $R / (\tau)$.
Bourbaki consider a similar condition in their discussion of Hermitian forms \cite{Bo59}, 
calling it \emph{Condition (T)}. Regarding \cref{eq:technical1}, note that $(\tau) = \overline{(\tau)}$
follows from $\ol{\,\cdot\,}$ being an automorphism. However, condition \cref{eq:technical1} is a slightly
stronger condition, which is e.g.~not satisfied by the DVR $\R[t^{\pm 1}]_{(t-1)}$.
The condition is only needed in the proof of \cref{lem:localpid}(iv).

Note that up to multiplication with units, $\tau$ is the unique irreducible element of $R$. It follows that every nonzero element in $R$ can be written as $r\tau^\ell$ for some $\ell \in \N$ and $r \in R$ a unit.
For any $R$-torsion module $A$, let $\nu\colon A\to \N$ be the function
sending $a\in A$ to the minimal $k \geq 0$ such that $\tau^k a = 0$.
Note that $\nu a = 0 \Leftrightarrow a = 0$.
The following lemma lays some technical groundwork for pairings over such rings $R$, and shows in particular that they are diagonalizable. By abuse of notation, we denote the functions $\nu\colon A\to \N$ and $\nu\colon Q(R)/R\to\N$
defined above both by the same symbol $\nu$.
\begin{lem}\label{lem:localpid}
Let $R$ be a DVR with involution satisfying \cref{eq:technical1} and \cref{eq:technical2}.
Let $A$ be a finitely generated $R$-torsion module and let
$\lambda\colon A\times A \to Q(R) / R$ be a sesquilinear Hermitian nonsingular pairing.
\begin{enumerate}[label=\emph{(\roman*)}]
\item If $B\subset A$ is a submodule such that $B\cap B^{\bot} = 0$, then $(A, \lambda)$
is isometric to the orthogonal sum of $B$ and $B^{\bot}$.

\item If $a \in A$ and $\nu a = \nu \lambda(a,a)$, then $(A,\lambda)$ is isometric to the orthogonal sum of $\langle a\rangle$ and $\langle a\rangle^{\bot}$.

\item For all $a\in A$, there exists $b\in A$ such that $\nu \lambda(a,b) = \nu a$.

\item There exists $a\in A$ with maximal $\nu a$ such that $\nu\lambda(a,a) = \nu a$.

\item $(A,\lambda)$ is isometric to an orthogonal sum $\langle e_1\rangle \oplus \cdots\oplus \langle e_n\rangle$ of cyclic modules. For any such $e_1, \ldots, e_n$, we have $\nu\lambda(e_i,e_i) = \nu e_i$ for all $i\in\{1,\ldots,n\}$.

\item Let $a\in A$ with $\nu a = 1$, and let $\tilde a\in A$ be an element with maximal $\nu \tilde a$ such that $a = \tau^{-1+\nu \tilde a}\tilde a$.
Then there exists $b \in A$ such that $\nu b = \nu \tilde a = \nu \lambda(\tilde a,b) = \nu \lambda(b,b)$.

\item
Let $a\in A$ and assume that $\nu \lambda(a,b) < \nu a$ for all $b \in A$ with $\nu b = \nu a$.
Then there exists $c\in A$ with $\nu c > \nu a$ such that $\nu (a + \tau c) < \nu a$.
\end{enumerate}
\end{lem}
\begin{proof}
(i) The sum of the inclusion maps gives a map
$B \oplus B^{\bot} \to A$
that clearly preserves pairings. It just remains to see that this map is bijective.
Injectivity follows from $B \cap B^{\bot} = \{0\}$.
Surjectivity follows since $\ord B \oplus B^{\bot} \doteq \ord A$ by \cref{lem:pid}~(ii), using that $R$ is a PID.
\smallskip

(ii) By (i), we just have to check
that $\langle a\rangle \cap \langle a\rangle^{\bot} = \{0\}$.
Take an element $ra \in \langle a\rangle$ with $r\in R$. If $ra\in\langle a\rangle^{\bot}$,
then $\lambda(ra, a) = 0$, and so $r\lambda(a,a) = 0$. Hence $r \in (\tau^{\nu\lambda(a,a)})
= (\tau^{\nu a})$, and thus $ra = 0$.
\smallskip

(iii) If $a = 0$, take $b = 0$. If $a \neq 0$, let $\tilde a = \tau^{-1 + \nu a} a$.
Since $\lambda$ is nonsingular, there exists $b\in A$ such that $\lambda(\tilde a, b) \neq 0$.
It follows that $\nu \lambda(\tilde a, b) = 1$ and $\nu \lambda(a,b) = \nu a$ as desired.
\smallskip

(iv) Pick any $b\in A$ of maximal $\nu b = k$.
If $\nu\lambda(b,b) = k$, just take $a = b$.
Else, by (iii) we may pick $c\in A$ such that $\nu \lambda(b,c) = k$, i.e.\ $\lambda(b,c) = r\tau^{-k}$ for $r\in R$ a unit.
Note this implies $\nu c = k$.
Again, if $\nu\lambda(c,c) = k$, just set $a = c$.
Else we have $\nu\lambda(b,b) < k$ and $\nu\lambda(c,c) < k$.
Pick $s\in R$ such that $s + \ol{s} = 1$ and set
$a = b + s\ol{r}^{-1}c$. This element $a$ satisfies $\nu \lambda(a,a) = k$, as one computes:
\begin{multline*}
\tau^{k-1}\lambda(b + s\ol{r}^{-1}c,b + s\ol{r}^{-1}c)  =
\tau^{k-1}(\lambda(b, s\ol{r}^{-1}c) + \lambda(s\ol{r}^{-1}c, b))
= \tau^{k-1}(\ol{s}\tau^{-k} + s\ol{\tau}^{-k}) = \tau^{-1}.
\end{multline*}

(v) Note that $\ord A \doteq \tau^j$ for some $j\geq 0$. The proof goes by induction over $j$.
For $j = 0$, $A$ is trivial and the statement holds.
If $A$ is nontrivial, by (iv), there exist a nontrivial $a\in A$ with maximal $\nu a$ and $\nu \lambda (a,a) = \nu a$.
By (ii), $A$ is isometric to $\langle a\rangle \oplus \langle a\rangle^{\bot}$, and by induction
$\langle a\rangle^{\bot}$ is isometric to an orthogonal sum of cyclic modules.

For the second statement, note that $\tau^{\nu e_i}\lambda(e_i, e_i) = 0$, so $\nu \lambda(e_i, e_i) \leq \nu e_i$.
Moreover, $\lambda$ is nonsingular, so there exists some $a\in A$ such that $\lambda(\tau^{\nu e_i - 1} e_i, a) \neq 0$.
Because of the orthogonality of the sum, we have $a = a_i e_i + b$ with $a_i\in R\setminus\{0\}$
and $b\in \langle e_1, \ldots, e_{i-1}, e_{i+1}, \ldots, e_n\rangle$. Then
$\lambda(\tau^{\nu e_i - 1} e_i, a) = a_i \tau^{\nu e_i - 1} \lambda(e_i, e_i) \neq 0$,
and so $\nu \lambda(e_i, e_i) \geq \nu e_i$.
\smallskip

(vi) Using (v), we may identify $(A, \lambda)$ with the orthogonal sum over cyclic modules with generators $e_1, \ldots, e_n$, so we write elements of $A$ as vectors with respect to that generator set.
In this notation, the $i$-th coordinate of $a$ is $r_i \tau^{-1+\nu e_i}$ for $r_i \in R$ a unit or 0.
Since $\nu a=1$, $a \neq 0$, and by maximality of $\nu \tilde{a}$, there exists $i$ such that $r_i \neq 0$ and $\nu e_i = \nu \tilde a$.
This implies that the $i$-th coordinate of $\tilde a$ is $r_i$.
Set $b$ to be $e_i$. Then $\nu b = \nu \tilde a$ and by the second part of (v), $\nu \lambda(e_i,e_i) = \nu e_i$, so
$\lambda(\tilde a,b) = \lambda(r_i e_i, e_i)$, where $\nu \lambda(r_i e_i, e_i) = \nu e_i$.
Likewise $\nu \lambda(b,b) = \nu \lambda(e_i, e_i) = \nu e_i$.%
\smallskip

(vii) Choose $e_1, \ldots, e_n$ as in (vi).
Let the $i$-th coordinate of $a$ be $r_i \tau^{k_i}$ with $r_i\in R$ a unit or $0$, and $k_i \geq 0$.
Reorder the $e_i$ such that there exists $j\in \{1, \ldots, n\}$ with
$\nu(r_i \tau^{k_i} e_i) = \nu(a)$ for $i \leq j$ and $\nu(r_i \tau^{k_i} e_i) < \nu(a)$ for $i > j$.
If $k_i = 0$ for some $i\in \{1,\ldots, j\}$, then
$\nu(a) = \nu(e_i) = \nu \lambda(e_i, a)$, contradicting the assumption that no such $b = e_i$ exists.
So $k_i \geq 1$ for all $i\in \{1,\ldots, j\}$.
Set
\[
c = -\sum_{i=1}^j r_i \tau^{k_i - 1} e_i.
\]
Then, clearly, $\nu c = 1 + \nu a$ and $\nu (a + \tau c) < \nu a$.
\end{proof}

Let $A$ be a finitely generated $R$-torsion module and
let $\lambda\colon A\times A \to Q(R) / R$ be a sesquilinear Hermitian nonsingular pairing.
By a technique known as \emph{dévissage} (see e.g.~\cite[Ch.~38]{Ra98}), the isometry type of $\lambda$ can be understood
by considering the isometry type of sesquilinear Hermitian nonsingular pairings over the field $\mathbb{F}\coloneqq R/(\tau)$. Let us make this precise now.

Let $A_k = \{a\in A\mid \tau^k a = 0\}$ for $k\geq 0$ and $A_k = \{0\}$ for $k < 0$.
Let
\[
\Phi_k(A) = A_k/(A_{k - 1} + \tau A_{k+1}).
\]
Note that $\lambda$ induces a sesquilinear Hermitian form $\Phi_k(\lambda)\colon \Phi_k(A) \times \Phi_k(A) \to (\tau^{-k} R) / (\tau^{-k+1} R)$, and $(\tau^{-k} R) / (\tau^{-k+1} R)$ is canonically identified with the field $\mathbb{F} = R/(\tau)$ by multiplication with $\tau^k$.
Nonsingularity of the form $\Phi_k(\lambda)$ follows from \cref{lem:localpid}~(vii).

Write $\mathfrak{T}(R)$ for the set of isometry classes of sesquilinear Hermitian nonsingular pairings $A\times A \to Q(R)/R$ on finitely generated $R$-torsion modules $A$, and $\mathfrak{M}(\mathbb{F})$ for the set of isometry classes of sesquilinear Hermitian nonsingular pairings $V\times V\to \mathbb{F}$ on finite dimensional vector fields $V$ over the field $R/(\tau)$. The orthogonal sum of pairings, which we denote by $\oplus$, makes both $\mathfrak{T}(R)$ and $\mathfrak{M}(\mathbb{F})$ into commutative monoids.
We may now see the $\Phi_i$ for $i\geq 1$ as a family of monoid homomorphisms $\mathfrak{T}(R) \to \mathfrak{M}(\mathbb{F})$ (abusing notation, let $\Phi_k(\lambda)$ denote both the concrete form induced by $\lambda$, as well as the isometry type of that form).

The monoid $\mathfrak{M}(\mathbb{F})$ enjoys some special properties,
due to condition \cref{eq:technical2} \cite[\S 4]{Bo59}.
Firstly, $\mathfrak{M}(\mathbb{F})$ is cancellative, i.e.\ $[\lambda] + [\mu_1] = [\lambda] + [\mu_2] \Rightarrow [\mu_1] = [\mu_2]$ for all $[\lambda],[\mu_1],[\mu_2] \in \mathfrak{M}(\mathbb{F})$ (this is the sesquilinear version of Witt's cancellation theorem).
Secondly, metabolic forms (defined as forms $V\times V\to\mathbb{F}$ admitting a metabolizer, i.e.~a subspace $P\subset V$ with $P^{\bot} = P$) are fully characterized by their dimension, and so
every metabolic form is a multiple of the
isometry class $H\in \mathfrak{M}(\mathbb{F})$ of the \emph{hyperbolic plane}, which is the pairing on
$\mathbb{F}^2$ sending $(e_1, e_1)$, $(e_2,e_2)$ to $0$ and $(e_1, e_2)$, $(e_2, e_1)$ to $1$.

The following lemma forms the core of the proof of \cref{cor:sig}.
\begin{lem}\label{lem:devissage}
Let $R$ be a DVR with involution satisfying \cref{eq:technical1} and \cref{eq:technical2}, and let  $A$ be a finitely generated $R$-torsion module with a sesquilinear Hermitian nonsingular pairing $\lambda$.
Let $G\subset A$ be a submodule such that $G\subset G^{\bot}$, and let $\lambda'$ be the pairing induced on $G^{\bot}/G$ by $\lambda$. Then there are $T_1, T_2, \ldots \in \mathfrak{M}(\mathbb{F})$ and $h_1, h_2, \ldots \in \N$
such that $T_1$ is trivial and for all $n\geq 1$, we have $h_{n} \geq h_{n+2}$ and
\[
T_n \oplus h_n \cdot H \oplus \bigoplus_{i=0}^{\infty} \Phi_{n+2i}(\lambda') = h_{n+1}\cdot H \oplus \bigoplus_{i=0}^{\infty} \Phi_{n+2i}(\lambda).
\]
\end{lem}

\begin{proof}
Throughout this proof, let us write $\lambda \geq_{\mathfrak M} \lambda'$ if $T_n$ and $h_n$ as above exist.
It is straightforward to check that the relation $\geq_{\mathfrak M}$ on $\mathfrak{M}(\mathbb{F})$ is transitive.

Now, $\lambda \geq_{\mathfrak M} \lambda'$ holds for trivial~$G$.
The main part of the proof is to show the statement for $\ord G \doteq \tau$.
Having established that, let $\ord G \doteq\tau^j$ with $j\geq 2$, and proceed by induction over $j$ as follows.
Pick $g\in G$ with $\nu g = 1$.
 Let $\mu$ be the pairing induced on $\langle g\rangle^{\bot} / \langle g\rangle$ by $\lambda$.
Note that the pairing induced on $(G / \langle g\rangle)^{\bot} / (G / \langle g\rangle)$ by $\mu$ is isometric to $\lambda'$. Since $\ord\langle g\rangle = \tau$ and $\ord G / \langle g\rangle = \tau^{j-1}$, we may apply the induction hypothesis twice, and get
$\lambda \geq_{\mathfrak M} \mu \geq_{\mathfrak M} \lambda'$.

So let us now deal with the case that $\ord G \doteq \tau$, i.e.\ $G = \langle g\rangle$ with $g\in \Phi_1(A)\setminus\{0\}$.
Let $\tilde g \in A$ have maximal $\nu \tilde g \eqqcolon k$ such that $g = \tau^{k-1}\tilde g$.
We will distinguish two cases, depending on whether
\begin{equation}\label{eq:case}
\nu \lambda(\tilde g, \tilde g) = k
\end{equation}
holds. If it does, set $B = \langle\tilde g\rangle$.
Otherwise, we have $\nu \lambda(\tilde g, \tilde g) < k$ and pick $b\in A$ as in \cref{lem:localpid}(vi), i.e.
\begin{equation}\label{eq:ordereq}
\nu b = \nu \lambda(\tilde g, b) = \nu \lambda(b,b) = k
\end{equation}
and set $B = \langle \tilde g, b\rangle\subseteq A$.
Let us check that $B \cap B^{\bot} = \{0\}$. In case \cref{eq:case} holds, this is true by \cref{lem:localpid}~(i).
If \cref{eq:case} does not hold, let $a \in B \cap B^{\bot}$ be given. Since $a\in B$, we have
$a = r \tau^{\ell} \tilde g + s \tau^m b$ for some units $r,s \in R$ and $\ell, m\geq 0$.
We want to show $a = 0$. For that, it is sufficient that $\ell \geq k$ and $m\geq k$, since both $\tilde g$ and $b$ are annihilated by $\tau^k$.
We have $\lambda(a, b) = 0$, since $a \in B^{\bot}$ and $b \in B$, and thus
\[
r \tau^{\ell} \lambda(\tilde g, b) + s \tau^m \lambda(b, b) = 0.
\]
If both summands are zero, then $\ell \geq k$ and $m \geq k$ using \cref{eq:ordereq}, and so we are done.
Otherwise, the two summands must have the same $\nu$, and so $k - \ell = k - m$, implying that $\ell = m$.
Now $\lambda(a, \tilde g) = 0$ implies that
\[
r \tau^{\ell} \lambda(\tilde g, \tilde g) + s \tau^m \lambda(b, \tilde g) = 0.
\]
If the second summand is zero, then $\ell \geq k$, and we are done.
Otherwise, the two summands must have the same $\nu$, i.e.~$\nu \lambda(\tilde g, \tilde g) - \ell = k - \ell$,
implying $\nu\lambda(\tilde g, \tilde g) = k$ contradicting the assumption that \cref{eq:case} does not hold.
We have thus shown $B \cap B^{\bot} = \{0\}$.

By \cref{lem:localpid}~(i), it follows that $A$ is isometric to $B \oplus B^{\bot}$.
Since $\langle g\rangle \subset B$, we have $B^{\bot} \subset \langle g\rangle^{\bot}$.
Thus $\langle g\rangle^{\bot} / \langle g\rangle$ is isometric to $C \oplus B^{\bot}$
where $C = (B \cap \langle g\rangle^{\bot}) / \langle g\rangle$
with the pairing induced by $\lambda$.
Hence it suffices to show $\lambda|_{B\times B} \geq_{\mathfrak{M}} \lambda|_{C\times C}$.

If \cref{eq:case} holds, then $B$ is cyclically generated by $\tilde g$ with $\nu\tilde g = k$,
where $k\geq 2$ follows.
So, $\Phi_n(B)$ is trivial for $n \neq k$, and $\Phi_k(\lambda_{B\times B})$ is isometric to the pairing on $\mathbb{F}$ sending $(e_1,e_1)\mapsto \tau^k\lambda(\tilde g, \tilde g) \in R / (\tau) = \mathbb{F}$.
Observe that $C$ is also cyclic, generated by $a = [\tau\tilde g]$, with $\nu a = k - 2$ and
$\lambda(a,a) = \tau^2 \lambda(\tilde g, \tilde g)$.
So $\Phi_n(C)$ is trivial for all $n\geq k-2$; and $\Phi_{k-2}(\lambda_{C\times C}) = \Phi_k(\lambda_{B\times B})$ if $k > 2$ (if $k = 2$, then $C$ is trivial). To demonstrate $\lambda|_{B\times B} \geq_{\mathfrak{M}} \lambda|_{C\times C}$,
one may take $h_n = 0$ for all $n$, $T_n = 0$ for all $n\neq k$,
and $T_k = \Phi_k(\lambda_{B\times B})$.

If \cref{eq:case} does not hold, recall that $B = \langle \tilde g, b\rangle$.
So $\Phi_n(B)$ is trivial for $n\neq k$ and $\Phi_k(B)$ is two-dimensional, with basis $[\tilde g], [b]$.
The pairing $\Phi_k(\lambda_{B\times B})$ is metabolic, since $([\tilde g], [\tilde g])$ is sent to $0$.
So $\Phi_k(\lambda_{B\times B}) = H$.
One calculates that $\langle g\rangle^{\bot} \cap B = \langle \tilde g, \tau b\rangle$,
so $\Phi_n(C)$ is trivial for $n \geq k - 1$; and $\Phi_{k-1}(\lambda_{C\times C}) = H$ if $k > 1$ (if $k = 1$, then $C$ is trivial).
To demonstrate $\lambda|_{B\times B} \geq_{\mathfrak{M}} \lambda|_{C\times C}$,
one may take $T_n = 0$ for all $n$, $h_n = 1$ if $n \leq k$ and $n \equiv k \pmod{2}$,
and $h_n = 0$ otherwise.
\end{proof}

We are now ready to prove \cref{cor:sig}.

\begin{proof}[Proof of \cref{cor:sig}]
If $e^{\pi i x}$ is not the root of the Alexander polynomial of any knot,
then $\sigma_x$ is a topological concordance invariant.
For such $x$, the stated inequalities thus simply read $0\geq 0\geq 0$.
So from now on, we may assume that $e^{\pi ix}$ is the root of some Alexander polynomial.
In fact, it will be sufficient to assume $e^{\pi i x}\not\in \R$ (equivalently, that $x\not\in \Z$).
Let $\zeta(t) = t^{-1} - 2\cos x + t$.
The localization $\R[t^{\pm 1}]_{(\zeta)}$ of $\R[t^{\pm 1}]$ at $(\zeta)$ is a DVR satisfying conditions \cref{eq:technical1} and \cref{eq:technical2} by \cref{lem:DVR-example}.

Note that $\R[t^{\pm 1}] / (\zeta)$ is isomorphic to $\mathbb{C}$ via the isomorphism sending $t$ to $e^{\pi ix}$.
So $\mathfrak{M}(\R[t^{\pm 1}]/(\zeta))$ is isomorphic to the monoid $\N^2$ via the signature.
Indeed, send $[\mu]\in \mathfrak{M}(\R[t^{\pm 1}]/(\zeta))$ to $(\mu_1, \mu_{-1}) \in \N^2$,
where $\mu_1$ and $\mu_{-1}$
are the dimensions of maximal subspaces on which $\mu$ is positive definite and negative definite, respectively.
Let us write $\sigma(\mu) = \mu_1 - \mu_{-1}$.

For a knot $L$, let us abbreviate $\Psi_i(L) \coloneqq \Phi_i(\Bl^{\R[t^{\pm 1}]_{(\zeta)}}_L)$.
Denote the difference between the signature and the averaged signature at~$e^{\pi ix}$ by $d_x(L)$:
\[
d_x(L) = \sigma_x(L) - \tfrac12\lim_{\varepsilon\to 0} \sigma_{x+\varepsilon}(L) + \sigma_{x-\varepsilon}(L).
\]
Note that $d_x(J) - d_x(K) = \sigma_x(J) - \sigma_x(K)$, since $\sigma_{x\pm\varepsilon}(J) = \sigma_{x\pm\varepsilon}(K)$
for almost all~$\varepsilon$.
One may read $\deg_x(L), \eta_x(L)$, and $d_x(L)$ from $\Psi_i(L)$
(see \cite{Lev89}, where the first and second equations are implicit,
and the third equation is given in Theorem~2.3):
\[
\deg_x(L) = \sum_{i=1}^\infty i\cdot \dim \Psi_i(L),\quad
\eta_x(L) = \sum_{i=1}^\infty \dim \Psi_i(L),\quad
d_x(L) = -\sum_{i=1}^\infty \sigma(\Psi_{2i}(L)).
\]
For the two knots $J, K$ with $[J]\geq_S [K]$, \cref{prop:appmain}
yields a submodule $\tilde G \subset H_1(X_J; \R[t^{\pm 1}]_{(\zeta)})$ such that the pairing induced on $\tilde G^{\bot} / \tilde G$ by $\Bl^{\R[t^{\pm 1}]_{(\zeta)}}_J$ is isometric to $\Bl^{\R[t^{\pm 1}]_{(\zeta)}}_K$. By \cref{lem:devissage}, this implies that for all $n\geq 1$,
there are $T_n\in \mathfrak{M}(\mathbb{F})$ and $h_n\in \N$ such that $T_1$ is trivial and $h_{n} \geq h_{n+2}$, and
\[
T_n \oplus h_n \cdot H \oplus \bigoplus_{i=0}^{\infty} \Psi_{n+2i}(K) = h_{n+1}\cdot H \oplus \bigoplus_{i=0}^{\infty} \Psi_{n+2i}(J).
\]
For $n = 2$, this implies
\begin{equation}\label{eq:t2}
\sigma(T_2) + \sum_{i=1}^{\infty} \sigma(\Psi_{2i}(K)) = \sum_{i=1}^{\infty} \sigma(\Psi_{2i}(J)).
\end{equation}
It also implies
\begin{equation}\label{eq:tn}
T_n \oplus T_{n+1} \oplus (h_n - h_{n+2})\cdot H \oplus \bigoplus_{i=0}^{\infty} \Psi_{n+i}(K) = \bigoplus_{i=0}^{\infty} \Psi_{n + i}(J)
\end{equation}
for all $n\geq 1$.
It follows that
\begin{multline*}
|\sigma_x(J) - \sigma_x(K)| =
|d_x(J) - d_x(K)|  = \Bigl|\sum_{i=1}^{\infty} \sigma(\Psi_{2i}(J)) - \sigma(\Psi_{2i}(K))\Bigr|
  \stackrel{\cref{eq:t2}}{=} |\sigma(T_2)| \\ \leq \dim(T_2)
  \leq \dim(T_2 + (h_1 - h_3)\cdot H)
  \stackrel{\cref{eq:tn}}{=} \sum_{i=1}^{\infty} \dim \Psi_i(J) - \dim \Psi_i(K) = \eta_x(J) - \eta_x(K),
\end{multline*}
proving the second desired inequality. To prove the first inequality:
\begin{align*}
\bigoplus_{n=2}^{\infty}\Bigl(
T_n \oplus T_{n+1} \oplus (h_n - h_{n+2})\cdot H \oplus \bigoplus_{i=0}^{\infty} \Psi_{n+i}(K)\Bigr) & \stackrel{\cref{eq:tn}}{=} \bigoplus_{n=2}^{\infty}\bigoplus_{i=0}^{\infty} \Psi_{n + i}(J) \\
\Rightarrow\quad (h_2 + h_3)\cdot H \oplus T_2 \oplus \bigoplus_{i=3}^{\infty}2T_i \oplus \bigoplus_{i=1}^{\infty} (i - 1)\Psi_{i}(K) & = \bigoplus_{i=1}^{\infty} (i - 1)\Psi_{i}(J) \\
\Rightarrow\quad \sum_{i=1}^{\infty} (i - 1)\dim \Psi_{i}(K) & \leq \sum_{i=1}^{\infty} (i - 1)\dim\Psi_{i}(J) \\[1.5ex]
\Rightarrow\quad \eta_x(J) - \eta_x(K) & \leq \deg_x(J) - \deg_x(K).
\qedhere
\end{align*}
\end{proof}

\section{Injections and surjections of twisted Alexander modules} \label{sec_TAM}

We provide an obstruction for the homotopy ribbon concordance in terms of twisted Alexander modules, which generalizes the results in \cite{FP19} on the classical Alexander module.

Let $J$ and $K$ be oriented knots in $S^3$ with $J \geq_{\top} K$, and let $C$ be a homotopy ribbon concordance from $J$ to $K$.
As in the proof of \cref{thm_blanchfield},
we denote the complement of an open tubular neighborhood of $C$ in $S^3 \times [0, 1]$ by $X_C$, and set $\pi_C = \pi_1(X_C)$.
We think of $X_J$ and $X_K$ naturally as subspaces of $\partial X_C$.
We write $\iota_J \colon \pi_J \to \pi_C$ and $\iota_K \colon \pi_K \to \pi_C$ for the induced epimorphism and the induced monomorphism respectively.
Let $\phi_C \colon \pi_C \to \Z$ be the abelianization map satisfying $\phi_J = \phi_C \circ \iota_J$ and $\phi_K = \phi_C \circ \iota_K$.

From here on out, let $R$ be a Noetherian UFD with (possibly trivial) involution $\ol{\,\cdot\,} \colon R \to R$, and $Q(R)$ is its quotient field.

\subsection{Surjections of twisted Alexander modules}

First we prove the following proposition, which generalizes \cite[Proposition 3.1]{FP19}.

\begin{prop} \label{prop_surjection}
Let $C$ be a homotopy ribbon concordance from $J$ to $K$ and $\alpha \colon \pi_C \to \GL(n, R)$ a representation.
Then the induced homomorphism
\[ H_1^{\alpha \circ \iota_J \otimes \phi_J}(X_J; R[t^{\pm 1}]^n) \to H_1^{\alpha \otimes \phi_C}(X_C; R[t^{\pm 1}]^n) \]
is surjective.
\end{prop}

The proof of \cref{prop_surjection} is parallel to the second proof of \cite[Proposition 3.1]{FP19}.
We need the following lemmas.

\begin{lem} \label{lem_vanishing}
For a homotopy ribbon concordance $C$ from $J$ to $K$, the following hold:
\[ H_0(X_C, X_J; \Z \pi_C) = H_1(X_C, X_J; \Z \pi_C) = 0. \]
\end{lem}

\begin{proof}
We consider the homology long exact sequence for $(X_C, X_J)$:
\[
\begin{split}
& H_1(X_C; \Z \pi_C)  \to H_1(X_C, X_J; \Z \pi_C)  \to \\
H_0(X_J; \Z \pi_C)  \to ~& H_0(X_C; \Z \pi_C)  \to H_0(X_C, X_J; \Z \pi_C)  \to 0.
\end{split}
\]
Since $\iota_J$ is surjective, the induced homomorphism $H_0(X_J; \Z \pi_C) \to H_0(X_C; \Z \pi_C)$ is an isomorphism.
Also, we have $H_1(X_C; \Z \pi_C) = H_1(\widetilde{X}_C; \Z) = 0$, which proves the lemma.
\end{proof}

\begin{lem} \label{lem_h1}
For a homotopy ribbon concordance $C$ from $J$ to $K$ and a representation $\alpha \colon \pi_C \to \GL(n, R)$, the following holds:
\[ H_1^{\alpha \otimes (\pm \phi_C)}(X_C, X_J; R[t^{\pm 1}]^n) = 0. \]
\end{lem}

\begin{proof}
We consider the universal coefficient spectral sequence for homology groups~\cite[Theorem 10.90]{Ro09}:
\[ E_{p, q}^2 = \tor_p^{\Z \pi_C}(H_q(X_C, X_J; \Z \pi_C), R[t^{\pm 1}]^n) \Rightarrow H_{p+q}^{\alpha \otimes (\pm \phi_C)}(X_C, X_J; R[t^{\pm 1}]^n), \]
where the $\tor$ group on the left hand side uses $\alpha \otimes (\pm \phi_C)$.
It follows from \cref{lem_vanishing} that
\[ \tor_1^{\Z \pi_C}(H_0(X_C, X_J; \Z \pi_C), R[t^{\pm 1}]^n) = H_1(X_C, X_J; \Z \pi_C) \otimes_{\Z \pi_C} R[t^{\pm 1}]^n = 0. \]
Thus all the terms on the line $p+q=1$ of $E^2$ page of the spectral sequence vanish, and the desired equation holds.
\end{proof}

The case of $-\phi_C$ in \cref{lem_h1} will be considered in the proof of \cref{lem_h123}.

\begin{proof}[Proof of \cref{prop_surjection}]
Consider the long exact sequence of pairs, and apply~\cref{lem_h1}.
\end{proof}

\subsection{Injections of twisted Alexander modules}

Next we prove the following proposition, which generalizes \cite[Proposition 3.4]{FP19}.

\begin{prop} \label{prop_injection}
Let $C$ be a homotopy ribbon concordance from $J$ to $K$ and let $\alpha \colon \pi_C \to \GL(n, R)$ be a representation.
If $\Delta_J^{\alpha \circ \iota_J} \neq 0$, then the induced homomorphism
\[ H_1^{\alpha \circ \iota_K \otimes \phi_K}(X_K; R[t^{\pm 1}]^n) \to H_1^{\alpha \otimes \phi_C}(X_C; R[t^{\pm 1}]^n) \]
is injective.
\end{prop}

We collect ingredients of the proof as the following three lemmas. The proof of the first one is straightforward.

\begin{lem} \label{lem_h0}
For a homotopy ribbon concordance $C$ from $J$ to $K$ and a representation $\alpha \colon \pi_C \to \GL(n, R)$, the following hold:
\[ H_0^{\alpha \otimes (\pm \phi_C)}(X_C, X_J; R[t^{\pm 1}]^n) = H_0^{\alpha \otimes (\pm \phi_J)}(X_C, X_K; R[t^{\pm 1}]^n) = 0. \]
\end{lem}

For a representation $\alpha \colon \pi_C \to \GL(n, R)$, we write $\alpha^\dagger \colon \pi_C \to \GL(n, R)$ for its dual given by
\[ \alpha(\gamma) = \ol{\alpha(\gamma^{-1})}^T \]
for $\gamma \in \pi_C$, where $\ol{\alpha(\gamma^{-1})}$ is the matrix obtained by replacing each entry $a$ in the matrix~$\alpha(\gamma^{-1})$ by~$\bar{a}$.
Note that $(\alpha^\dagger)^\dagger = \alpha$.

\begin{lem} \label{lem_ss}\leavevmode
\begin{enumerate}[label=\emph{(\roman*)}]
\item There exists a convergent spectral sequence
\[ E_2^{p, q} = \ext_q^{R[t^{\pm 1}]}(H_p^{\alpha \otimes \phi_C}(X_C, X_K; R[t^{\pm 1}]^n), R[t^{\pm 1}]) \Rightarrow H_{4-p-q}^{\alpha^\dagger \otimes (-\phi_C)}(X_C, X_J; R[t^{\pm 1}]^n). \]
\item There exists a convergent spectral sequence
\[ E_2^{p, q} = \ext_q^{R[t^{\pm 1}]}(H_p^{\alpha^\dagger \otimes (-\phi_C)}(X_C, X_J; R[t^{\pm 1}]^n), R[t^{\pm 1}]) \Rightarrow H_{4-p-q}^{\alpha \otimes \phi_C}(X_C, X_K; R[t^{\pm 1}]^n). \]
\end{enumerate}
\end{lem}

\begin{proof}
Since the proof of (i) is similar to that of (ii), we only prove (ii).
By Poincar\'e-Lefschetz duality we have an isomorphism
\[ H_{4-i}^{\alpha \otimes \phi_C}(X_C, X_K; R[t^{\pm 1}]^n) \cong H_{\alpha \otimes \phi_C}^i(X_C, X_J; R[t^{\pm 1}]^n) \]
(See for instance \cite[Theorem A.15]{FNOP19} for the case of twisted coefficients.)
Also, we have an isomorphism
\[ \hom_{\Z \pi_C}(C_*(\widetilde{X}_C, \widetilde{X}_J), R[t^{\pm 1}]^n) \cong \hom_{R[t^{\pm 1}]}(C_*(\widetilde{X}_C, \widetilde{X}_J) \otimes_{\Z \pi_C} R[t^{\pm 1}]^n, R[t^{\pm 1}]), \]
where the tensor product on the right hand side uses $\alpha^\dagger \otimes (-\phi_C)$, given by
\[ f \mapsto \left( (c \otimes v) \mapsto f(c)^T \bar{v} \right) \]
for $f \in \hom_{\Z \pi_C}(C_*(\widetilde{X}_C, \widetilde{X}_J), R[t^{\pm 1}]^n)$, $c \in C_*(\widetilde{X}_C, \widetilde{X}_J)$ and $v \in R[t^{\pm 1}]^n$.
Now the lemma follows from the universal coefficient spectral sequence for cohomology groups~\cite[Theorem 2.3]{Lev77}.
\end{proof}

\begin{lem} \label{lem_h123}
Let $C$ be a homotopy ribbon concordance from $J$ to $K$ and let $\alpha \colon \pi_C \to \GL(n, R)$ be a representation.
If $\Delta_J^{\alpha \circ \iota_J} \neq 0$, then the following hold:
\begin{enumerate}[label=\emph{(\roman*)}]
\item $H_1^{\alpha \otimes \phi_C}(X_C, X_K; R[t^{\pm 1}]^n)$ is a torsion module.
\item $H_2^{\alpha^\dagger \otimes (-\phi_C)}(X_C, X_J; R[t^{\pm 1}]^n)$ is a torsion module.
\item $H_3^{\alpha^\dagger \otimes (-\phi_C)}(X_C, X_J; R[t^{\pm 1}]^n) = 0$.
\end{enumerate}
\end{lem}

\begin{proof}
We first prove (i).
We consider the homology long exact sequence for $(X_C, X_K)$:
\[ H_1^{\alpha \otimes \phi_C}(X_C; R[t^{\pm 1}]^n) \to H_1^{\alpha \otimes \phi_C}(X_C, X_K; R[t^{\pm 1}]^n) \to H_0^{\alpha \circ \iota_K \otimes \phi_K}(X_K; R[t^{\pm 1}]^n). \]
It is easy to check that $H_0^{\alpha \circ \iota_K \otimes \phi_K}(X_K; R[t^{\pm 1}]^n)$ is a torsion module.
Since $\Delta_J^{\alpha \circ \iota_J} \neq 0$, the module $H_1^{\alpha \circ \iota_J \otimes \phi_J}(X_J; R[t^{\pm 1}]^n)$ is torsion, and it follows from \cref{prop_surjection} that $H_1^{\alpha \otimes \phi_C}(X_C; R[t^{\pm 1}]^n)$ is also a torsion module.
Therefore (i) follows.

Next we prove (iii).
It follows from \cref{lem_h0} and (i) that
\begin{align*}
\ext_1^{R[t^{\pm 1}]}(H_0^{\alpha \otimes \phi_C}(X_C, X_K; R[t^{\pm 1}]^n), R[t^{\pm 1}]) &= 0, \\
\hom_{R[t^{\pm 1}]}(H_1^{\alpha \otimes \phi_C}(X_C, X_K; R[t^{\pm 1}]^n), R[t^{\pm 1}]) &=0.
\end{align*}
Thus all the terms on the line $p+q=1$ of $E^2$ page of the spectral sequence of \cref{lem_ss}~(i) vanish.
Therefore (iii) follows.

Finally we prove (ii).
Let $(X, Y)$ be a finite CW-pair homotopy equivalent to $(X_C, X_J)$, and we identify $\pi_1(X)$ with $\pi_C$.
By (iii) and \cref{lem_h1,lem_h0}, we have
\begin{align*}
&\rank_{R[t^{\pm 1}]} H_2^{\alpha^\dagger \otimes (-\phi_C)}(X_C, X_J; R[t^{\pm 1}]^n)\\
=& \sum_{i=0}^3 (-1)^i \rank_{R[t^{\pm 1}]} H_i^{\alpha^\dagger \otimes (-\phi_C)}(X_C, X_J; R[t^{\pm 1}]^n) \\
=& \sum_{i=0}^3 (-1)^i \rank_{R[t^{\pm 1}]} H_i^{\alpha^\dagger \otimes (-\phi_C)}(X, Y; R[t^{\pm 1}]^n) \\
=& \sum_{i=0}^3 (-1)^i \rank_{R[t^{\pm 1}]} \ol{C}_i(\widetilde{X}, \widetilde{Y}) \otimes_{\Z \pi_C} R[t^{\pm 1}]^n \\
=& n \sum_{i=0}^3 (-1)^i \rank_\Z C_i(X, Y) \\
 =& n \chi(X, Y)
= n \chi(X_C, X_J)
= n \left( \chi(X_C) - \chi(X_J) \right) = 0,
\end{align*}
where we consider the cellular chain complexes $C_*(X, Y)$ and $C_*(\widetilde{X}, \widetilde{Y})$.
Therefore (ii) follows, which completes the proof.
\end{proof}

We are now in a position to prove \cref{prop_injection}.
\begin{proof}[Proof of \cref{prop_injection}]
It follows from \cref{lem_h1,lem_h0,lem_h123}(ii) that
\begin{align*}
\ext_2^{R[t^{\pm 1}]}(H_0^{\alpha^\dagger \otimes (-\phi_C)}(X_C, X_J; R[t^{\pm 1}]), R[t^{\pm 1}]^n) &=0, \\
\ext_1^{R[t^{\pm 1}]}(H_1^{\alpha^\dagger \otimes (-\phi_C)}(X_C, X_J; R[t^{\pm 1}]), R[t^{\pm 1}]^n) &=0, \\
\hom_{R[t^{\pm 1}]}(H_2^{\alpha^\dagger \otimes (-\phi_C)}(X_C, X_J; R[t^{\pm 1}]), R[t^{\pm 1}]^n) &=0.
\end{align*}
Thus all the terms on the line $p+q=2$ of $E^2$ page of the spectral sequence of \cref{lem_ss}(ii) vanish.
Therefore $H_2^{\alpha \otimes \phi_C}(X_C, X_K; R[t^{\pm 1}]^n) = 0$, and the proposition follows.
\end{proof}

\begin{rem}
In the proofs of \cref{prop_surjection,prop_injection} we did not use the injectivity of $\iota_K$.
\end{rem}

\subsection{Divisibility of twisted Alexander polynomials}

Now \cref{thm_division} is a corollary of \cref{prop_surjection,prop_injection}.
For the readers' convenience we recall the statement.
\thmdivision*

\begin{proof}
We may suppose that $\Delta_J^{\alpha \circ \iota_J} \neq 0$, since $g~|~0$ for every $g \in R[t^{\pm 1}]$, so if $\Delta_J^{\alpha \circ \iota_J} = 0$ then the theorem trivially holds.
Let $\Delta_C^{\alpha}(t)$ be the order of $H_1^{\alpha \otimes \phi_C}(X_C; R[t^{\pm 1}]^n)$.
Since orders are multiplicative for a short exact sequence of finitely generated $R[t^{\pm 1}]$-modules, it follows from \cref{prop_surjection,prop_injection} that
\[ \Delta_C^{\alpha} ~|~ \Delta_J^{\alpha \circ \iota_J} ~\text{and}~ \Delta_K^{\alpha \circ \iota_K} ~|~ \Delta_C^{\alpha}, \]
which proves the theorem.
\end{proof}

\begin{rem}
        In order to use the results of this section as a homotopy ribbon concordance obstruction, one needs to find a representation~$\alpha \colon \pi_1(X_C) \to \GL(n, R)$ for a putative concordance~$C$, by defining representations~$\pi_1(X_K) \to \GL(n,R)$ and $\pi_1(X_J) \to \GL(n,R)$ that extend over \emph{any} homotopy concordance exterior. Of course, we can use representations that extend over any topological concordance exterior: while abelian representations recover the classical obstructions~\cite{FP19}, we will illustrate the use of metabelian representations in \cref{sec_proof}, making use of the assumption that the Alexander modules of~$K$ and~$J$ coincide.
\end{rem}

\section{A satellite construction} \label{sec_satellite}

We describe a satellite construction of knots, which produce an example of an infinite family of knots as in \cref{thm_main}.

\subsection{Satellite knots} \label{subsec_satellite}

Let $J$ and $K$ be oriented knots in $S^3$, and let $A$ be a simple closed curve in $X_K$ unknotted in $S^3$.
We think of $K$ also as a simple closed curve in $X_A$.
Let $\phi \colon \partial X_A \to \partial X_J$ be a diffeomorphism sending a longitude of $A$ to a meridian of~$J$, and a meridian of $A$ to a longitude of $J$.
We identify $X_A \cup_\phi X_J$ with $S^3$, and denote  the image of $K$ by $S = S(K, J, A)$.
The knot $S$ inherits an orientation from~$K$.
We call $S$ the \emph{satellite knot with companion $J$, orbit $K$ and axis $A$}.
In other words, $S$ is the satellite knot of $J$ with pattern $K$ in the solid torus $X_A$.
If $A$ bounds a disc in $S^3$ pierced by $w$ strands of $K$, then $S$ is isotopic to the knot obtained by replacing the $w$ trivial strands by $w$ untwisted parallels of a knotted arc whose oriented knot type is $J$ as in Figure~\ref{fig_satellite}.

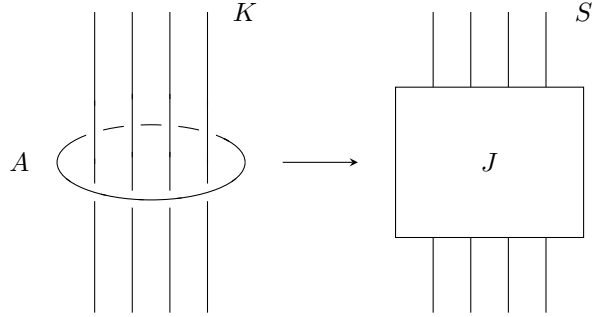
\begin{figure}[htb]
\centering
\begin{tikzpicture}
\begin{knot}[consider self intersections=true, ignore endpoint intersections=false, clip width=4, flip crossing/.list={3,4,5,6}]
\strand (-2.25,0) circle [x radius=1.25, y radius=0.5];
\strand (-3,-2)--(-3,2);
\strand (-2.5,-2)--(-2.5,2);
\strand (-2,-2)--(-2,2);
\strand (-1.5,-2)--(-1.5,2);
\end{knot}
\draw[->,>=stealth] (-.5,0)--(.5,0);
\draw (1,-1) rectangle (3.5,1);
\draw (1.5,-2)--(1.5,-1);
\draw (1.5,1)--(1.5,2);
\draw (2,-2)--(2,-1);
\draw (2,1)--(2,2);
\draw (2.5,-2)--(2.5,-1);
\draw (2.5,1)--(2.5,2);
\draw (3,-2)--(3,-1);
\draw (3,1)--(3,2);
\node (A) at (-4,0) {$A$};
\node (K) at (-1,2) {$K$};
\node (J) at (2.25,0) {$J$};
\node (S) at (3.5,2) {$S$};
\end{tikzpicture}
\caption{The satellite knot $S(K, J, A)$ in the case $w=4$}
\label{fig_satellite}
\end{figure}

Let $\nu_A$ be an open tubular neighborhood of $A$ in $X_K$.
The abelianization map $\pi_J \to \Z$ gives rise to a degree one map $X_J \to \ol{\nu_A}$ that is a diffeomorphism on the boundary.
It defines a map
\[ X_S = \left( X_K \setminus \nu_A \right) \cup_\phi X_J \to \left( X_K \setminus \nu_A \right) \cup_\phi \ol{\nu_A} = X_K. \]
Let $\rho_S \colon \pi_S \to \pi_K$ be the induced epimorphism.
Note that $A$ determines an element $[A] \in \pi_K$ up to conjugation.

For a nonnegative integer $m$, we write $\pi_K^{(m)}$ for the $m$-th derived subgroup of $\pi_K$, defined inductively by
\[ \pi_K^{(0)} = \pi_K ~\text{and}~ \pi_K^{(m+1)} = [\pi_K^{(m)}, \pi_K^{(m)}]. \]
The following lemma was proved in \cite[Theorem 8.1]{Coc04}.

\begin{lem} \label{lem_infection}
Let $S = S(K, J, A)$ be the satellite knot with companion $J$, orbit $K$ and axis $A$.
If $[A] \in \pi_K^{(m)}$, then $\rho_S$ induces an isomorphism
\[ \pi_S / \pi_S^{(m+1)} \to \pi_K / \pi_K^{(m+1)}. \]
\end{lem}

The following lemmas are immediate consequences of the construction.
We record them for use in the proof of \cref{thm_main}.
\begin{lem} \label{lem_seifert}
Let $V$ be a Seifert form of $K$ associated to a Seifert surface $F$, and let $S = S(K, J, A)$ be the satellite knot with companion $J$, orbit $K$ and axis $A$.
If $A$ is disjoint from $F$, then $S$ has a Seifert form isomorphic to $V$.
\end{lem}

\begin{proof}
Since $A$ is disjoint from $F$, we have the image $F'$ of $F$ by the inclusion map $X_A \to X_S$.
We see at once that $F'$ is a Seifert surface of $S$, and its Seifert form is isomorphic to $V$, which proves the lemma.
\end{proof}

\begin{lem} \label{lem_concordance}
Let $S = S(K, J, A)$ be the satellite knot with companion $J$, orbit $K$ and axis $A$.
\begin{enumerate}[label=\emph{(\roman*)}]
  \item If $J$ is a ribbon knot, then $S$ is ribbon concordant to $K$.
  \item If $J$ is a homotopy ribbon knot, then $S$ is homotopy ribbon concordant to $K$.
\end{enumerate}
\end{lem}

\begin{proof}
First we prove (i). A ribbon concordance from $J$ to the unknot can be encoded as a sequence of $n$ saddle moves from $J$ to the unlink with $n+1$ components. If $A$ bounds a disc in $S^3$ intersecting $K$ in $w$ points, then one may find a sequence of $wn$ saddle moves from $S$ to the disjoint union of $K$ with a $wn$-component unlink.
To find them, for each saddle move on $J$, perform $w$ parallel copies of it on~$S$ as in Figure~\ref{fig_saddle}. This converts $S$ to a disjoint union of $K$ and a $wn$-component unlink. Cap off the unlink with $wn$ discs to complete the construction of a ribbon concordance from~$S$ to~$K$. This completes the proof of (i).

\begin{figure}[htb]
\centering
\begin{tikzpicture}
\begin{scope}[shift={(-3,0)}]
\draw[->,>=stealth] (-162:2) to [out=72, in=-72] (162:2);
\draw[->,>=stealth] (-144:2) to [out=72, in=-72] (144:2);
\draw[->,>=stealth] (-126:2) to [out=72, in=-72] (126:2);
\draw[->,>=stealth] (-108:2) to [out=72, in=-72] (108:2);
\draw[<-,>=stealth] (-72:2) to [out=108, in=-108] (72:2);
\draw[<-,>=stealth] (-54:2) to [out=108, in=-108] (54:2);
\draw[<-,>=stealth] (-36:2) to [out=108, in=-108] (36:2);
\draw[<-,>=stealth] (-18:2) to [out=108, in=-108] (18:2);
\end{scope}
\draw[->,>=stealth] (-.5,0)--(.5,0);
\begin{scope}[shift={(3,0)}]
\draw[->,>=stealth] (-108:2) to [out=18, in=162] (-72:2);
\draw[->,>=stealth] (-126:2) to [out=18, in=162] (-54:2);
\draw[->,>=stealth] (-144:2) to [out=18, in=162] (-36:2);
\draw[->,>=stealth] (-162:2) to [out=18, in=162] (-18:2);
\draw[->,>=stealth] (18:2) to [out=-162, in=-18] (162:2);
\draw[->,>=stealth] (36:2) to [out=-162, in=-18] (144:2);
\draw[->,>=stealth] (54:2) to [out=-162, in=-18] (126:2);
\draw[->,>=stealth] (72:2) to [out=-162, in=-18] (108:2);
\end{scope}
\end{tikzpicture}
\caption{Parallel copies of a saddle move in the case $w=4$}
\label{fig_saddle}
\end{figure}
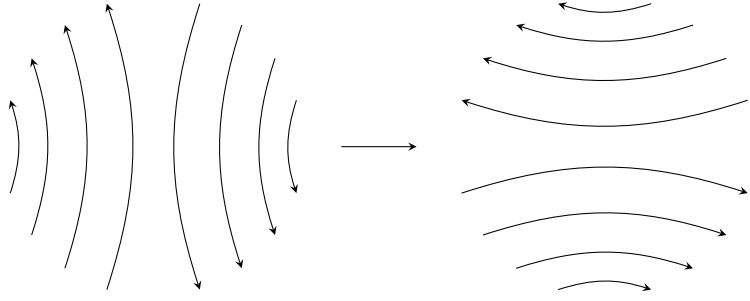

Now we prove (ii).
Since $J$ is a homotopy ribbon knot,  $J$ is homotopy ribbon concordant to the unknot $U$, via a homotopy ribbon concordance $C$. So the exterior $X_C$ has boundary $X_J \cup S^1 \times S^1 \times I \cup -X_U$.
Recall that $X_{S(K,J,A)} = X_{K  \cup A} \cup_{\partial \ol{\nu A}} X_J$ and that, trivially,  $X_K =  X_{K  \cup A} \cup_{\partial\ol{\nu A}} X_U$.
Glue to form:
\[Y:= X_{K \cup A} \times I  \cup_{\partial \ol{\nu A} \times I}  X_C\]
along the $S^1 \times S^1 \times I$ part of the boundary of $X_C$.  Note that $\partial Y = X_S \cup -X_K$ and $Y$ is a $\Z$--homology cobordism with $\pi_1(Y)$ normally generated by a meridian of $K$ or of $S$.
Therefore $Y \cup S^1 \times D^2 \times I$ is an 4-dimensional $h$-cobordism from $S^3$ to itself. The 5-dimensional $h$-cobordism theorem~\cite[Theorem~7.1A]{FQ90} may be applied to deduce that $Y \cup S^1 \times D^2 \times I \cong S^3 \times I$, and therefore $Y$ is a concordance exterior.  The image of $S^1 \times {0} \times I$ under this homeomorphism gives rise to a concordance from $S(K,J,A)$ to~$K$.
The fact that $X_C$ is a homotopy ribbon concordance exterior easily implies that $Y$ is too. This completes the proof of (ii) and of the lemma.
\end{proof}

\subsection{A formula for twisted Alexander polynomials}

The following proposition is proved by reinterpreting twisted Alexander polynomials as Reidemeister torsion (see \cite{KL99, Ki96}) and using a surgery formula for Reidemeister torsion.
For a proof of a more general statement we refer the reader to \cite[Lemma 7.1]{CF10}.

\begin{prop} \label{prop_satellite}
Let $S = S(K, J, A)$ be the satellite knot with companion $J$, orbit $K$ and axis $A$, and let $\alpha \colon \pi_K \to GL(n, \Z)$ be a representation.
Suppose that $A$ is null-homologous in $X_K$.
Then the following holds:
\begin{equation} \label{eq_satellite}
\Delta_S^{\alpha \circ \rho_S}(t) \doteq \Delta_K^\alpha(t) \prod_{i=1}^n \Delta_J(z_i),
\end{equation}
where $z_1, \dots, z_n$ are the eigenvalues of $\alpha([A])$.
\end{prop}

Recall that the resultant of two polynomials
\[ f(t) = c \prod_{i=1}^m \left( t - \alpha_i \right),~ g(t) = d \prod_{j=1}^n \left( t - \beta_j \right) \]
in $\C[t]$, where $c, d, \alpha_1, \dots, \alpha_m, \beta_1, \dots, \beta_n \in \C$, is given by
\[ \res(f, g) = c^n d^m \prod_{1 \leq i \leq m,~ 1 \leq j \leq n} \left( \alpha_i - \beta_j \right). \]

\begin{rem} \label{rem_resultant}
In \cref{prop_satellite}, taking a representative $\Delta_J \in \Z[t]$, we can rewrite \cref{eq_satellite} as
\[ \Delta_S^{\alpha \circ \rho_S}(t) \doteq \frac{1}{c^n} \res(\Delta_J, f_{\alpha([A])}) \Delta_K^\alpha(t), \]
where $c$ is the leading coefficient of $\Delta_J(t)$ and $f_{\alpha([A])}(t)$ is the characteristic polynomial of $\alpha([A])$.
\end{rem}

\section{Metabelian Alexander polynomials} \label{sec_MAP}

We introduce twisted Alexander polynomials associated to certain metabelian representations, which are useful for applications of \cref{thm_division}.
An important property of these invariants is that they are always nonzero.
(See \cite{HKL10} for a similar construction of twisted Alexander polynomials.)

\subsection{Metabelian representations}

Let $K$ be an oriented knot in $S^3$ and let $r$ be a positive integer.
Let $X_{K, r}$ and $X_{K, \infty}$ be the $r$-fold cyclic cover and the infinite cyclic cover of $X_K$ respectively, and let $\Sigma_{K, r}$ be the $r$-fold cyclic branched cover of $S^3$ along $K$.
We set $H_K = H_1(X_{K, \infty}; \Z)$.
The abelianization map $\phi_K \colon \pi_K \to \Z$ defines an action of $\Z$ on $X_{K, \infty}$, and we write $t$ for the automorphism on $H_K$ corresponding to $1 \in \Z$.
Then $H_K$ has the structure of a $\Z[t^{\pm 1}]$-module.

Let $p$ be a prime, and  let $\F_p$ be the field of order $p$.
We identify
\[ \pi_K / \pi_K^{(2)} = \pi_K^{(1)} / \pi_K^{(2)} \rtimes \pi_K / \pi_K^{(1)} = H_K \rtimes \Z. \]
We set $\Gamma_K^{r, p}$ to be the quotient group
\[\Gamma_K^{r, p} := \left( H_K \otimes_{\Z[t^{\pm 1}]} \F_p[t^{\pm 1}] / (t^r-1) \right) \rtimes \Z / r \Z, \]
and we write \[\ol{\alpha}_K^{r, p} \colon \pi_K / \pi_K^{(2)} \to \GL(|\Gamma_K^{r, p}|, \Z)\] for the composition of the quotient map and the regular representation $\Gamma_K^{r, p} \to \GL(|\Gamma_K^{r, p}|, \Z)$, fixing a basis of the free abelian group $\Z \Gamma_K^{r, p}$.
We define a representation \[\alpha_K^{r, p} \colon \pi_K \to \GL(|\Gamma_K^{r, p}|, \Z)\] to be the composition of the quotient map $\pi_K \to \pi_K / \pi_K^{(2)}$ and $\ol{\alpha}_K^{r, p}$, and set
\[ \Delta_K^{r, p}(t) = \Delta_K^{\alpha_K^{r, p}}(t) \in \Z[t^{\pm 1}]. \]
Since it is invariant under conjugation of $\alpha_K^{r, p}$ as in \cref{rem_TAP}~(2), it does not depend on the choice of a basis of $\Z \Gamma_K^{r, p}$.

It is straightforward to see that the composition $X_{K, \infty} \to \Sigma_{K, r}$ of the projection $X_{K, \infty} \to X_{K, r}$ and the inclusion map $X_{K, r} \to \Sigma_{K, r}$ induces an isomorphism
\[ H_K \otimes_{\Z[t^{\pm 1}]} \F_p[t^{\pm 1}] / (t^r-1) \cong H_1(\Sigma_{K, r}; \F_p). \]
Since $\ker \phi_K = \pi_K^{(1)}$ and $\pi_K^{(1)} / \pi_K^{(2)} = H_K$, we see that $\alpha_K^{r, p} |_{\ker \phi_K}$ factors through this $p$-group.

The following theorem is proved in~\cite[Theorem 1.2]{Liv02}.
This shows that if $\Delta_K \neq 1$, then there exist some $r$ and $p$ such that $\alpha_K^{r, p} |_{\ker \phi_K}$ is nontrivial.

\begin{thm} \label{thm_Liv02}
The $3$-manifold $\Sigma_{K, r}$ is a homology $3$-sphere for every $r$ if and only if $\Delta_K(t) = 1$.
\end{thm}

\begin{rem}
Livingston~\cite[Theorem 1.2]{Liv02} also showed that $\Sigma_{K, r}$ is a homology $3$-sphere for every prime power $r$ if and only if every nontrivial irreducible factor of $\Delta_K(t)$ is a cyclotomic polynomial $\Phi_n(t)$ for some $n$ divisible by three distinct primes.
\end{rem}

The following lemma will be useful in \cref{sec_proof}.

\begin{lem} \label{lem_MAP}
Let $J$ and $K$ be oriented knots in $S^3$ with an isomorphism $\psi \colon \pi_J / \pi_J^{(2)} \to \pi_K / \pi_K^{(2)}$.
Then for any positive integer $r$ and any prime number $p$,
\[ \Delta_J^{\psi^* \alpha_K^{r, p}}(t) \doteq \Delta_J^{r, p}(t), \]
where $\psi^* \alpha_K^{r, p}$ is the composition
\[\psi^* \alpha_K^{r, p} \colon \pi_J \to \pi_J / \pi_J^{(2)} \xrightarrow{\psi} \pi_K / \pi_K^{(2)} \xrightarrow{\ol{\alpha}_K^{r, p}} \GL(|\Gamma_K^{r, p}|, \Z).\]
\end{lem}

\begin{proof}
The groups $\Gamma_J^{r, p}$ and $\Gamma_K^{r, p}$ are quotients of $\pi_J / \pi_J^{(2)}$ and $\pi_K / \pi_K^{(2)}$ respectively, and $\psi$ induces an isomorphism $\Gamma_J^{r, p} \to \Gamma_K^{r, p}$.
Since the actions of $\pi_J / \pi_J^{(2)}$ on $\Gamma_J^{r, p}$ and $\pi_K / \pi_K^{(2)}$ on $\Gamma_K^{r, p}$ are compatible with $\psi$ and the induced isomorphism, $\ol{\alpha}_K^{r, p} \circ \psi$ and $\ol{\alpha}_J^{r, p}$ are conjugate, and so are $\psi^* \alpha_K^{r, p}$ and $\alpha_J^{r, p}$.
Therefore the lemma follows from \cref{rem_TAP}~(2).
\end{proof}

\subsection{Nonvanishing of metabelian Alexander polynomials}

The following theorem is significant in applying \cref{thm_division}.
\begin{thm} \label{thm_non-vanishing}
For any positive integer $r$ and any prime number $p$, $\Delta_K^{r, p} \neq 0$.
\end{thm}

For the proof we need the following theorem by the first and fifth authors~\cite[Theorem 3.1]{FP12}.

\begin{thm} \label{thm_FP12}
Let $\pi$ be a group, $p$ a prime and $f \colon M \to N$ a morphism of projective right $\Z \pi$-modules such that the homomorphism
\[ M \otimes_{\Z \pi} \F_p \to N \otimes_{\Z \pi} \F_p \]
induced by $f$ is injective.
Let $\phi \colon \pi \to H$ be an epimorphism onto a torsion-free abelian group and let $\alpha \colon \pi \to \GL(n, Q)$ be a representation over a field $Q$ of characteristic~$0$.
If $\alpha|_{\ker \phi}$ factors through a $p$-group, then the homomorphism
\[ M \otimes_{\Z \pi} Q(H)^n \to N \otimes_{\Z \pi} Q(H)^n \]
induced by $f$ is also injective, where $Q(H)$ is the quotient field of the group ring of $H$ over $Q$.
\end{thm}

\begin{rem}
The original result~\cite[Theorem 3.1]{FP12} concerns projective left $\Z \pi$-modules.
But we can think of them also as right $\Z \pi$-modules, using the involution of $\Z \pi$ reversing elements of $\pi$, and the results of tensor products from the left are naturally identified with those from the right.
Thus we adapt the statement to the case of projective right $\Z \pi$-modules.
\end{rem}

\begin{proof}[Proof of \cref{thm_non-vanishing}]
The proof is similar to that of \cite[Proposition 4.1]{FP12}.

Let $\mu$ be a meridian in $X_K$.
Since the induced homomorphism $H_*(\mu; \F_p) \to H_*(X_K; \F_p)$ is an isomorphism, we have $H_*(X_K, \mu; \F_p) = 0$.
We first prove that $H_*^{\alpha_K^{r, p} \otimes \phi_K}(X_K, \mu; \Q(t)^n) = 0$, following a standard argument of chain homotopy lifting (see for instance \cite[Proposition 2.10]{COT03}).

We pick a CW complex structure of $X_K$ containing $\mu$ as a subcomplex, and set $C_*$ and $\widetilde{C}_*$ to be the cellular chain complexes $C_*(X_K, \mu) \otimes_\Z \F_p$
and $C_*(\widetilde{X}_K, \widetilde{\mu})$ respectively, where $\widetilde{\mu}$ is the pullback of $\mu$ in the universal cover $\widetilde{X}_K$ of $X_K$.
Let $h_i \colon C_i \to C_{i+1}$ be a chain contraction for $C_*$, so that $\partial h_i + h_{i-1} \partial$ is the identity map on $C_i$.
Since the natural map $\widetilde{C}_* \to C_*$ is surjective and all modules in $C_*$ and $\widetilde{C}_*$ are free, there exists a lift $\tilde{h}_i \colon \widetilde{C}_i \to \widetilde{C}_{i+1}$ of $h_i$ for each $i$.
We write $f_i \colon \widetilde{C}_i \to \widetilde{C}_i$ for the chain map $\partial \tilde{h}_i + \tilde{h}_{i-1} \partial$.

Since the chain map $C_i \to C_i$ induced by $f_i$ is the identity map and $\alpha_K^{r, p}|_{\ker \phi_K}$ factors through the $p$-group
\[ H_K \otimes_{\Z[t^{\pm 1}]} \F_p[t^{\pm 1}] / (t^r-1), \]
it follows from \cref{thm_FP12} that the chain map
\[ \widetilde{C}_i \otimes_{\Z \pi_K} \Q(t)^n \to \widetilde{C}_i \otimes_{\Z \pi_K} \Q(t)^n \]
induced by $f_i$ is injective, and is therefore an isomorphism.
Thus $\tilde{h}_i$ induces a chain homotopy between the zero map and the chain isomorphism.
Therefore $H_*^{\alpha_K^{r, p} \otimes \phi_K}(X_K, \mu; \Q(t)^n) = 0$.

By a direct computation we have $H_*^{(\alpha_K^{r, p} \otimes \phi_K) \circ \iota}(\mu; \Q(t)^n) = 0$, where $\iota \colon \pi_1(\mu) \to \pi_K$ is the induced homomorphism.
Now it follows from the homology long exact sequence for $(X_K, \mu)$ that $H_*^{\alpha_{r, p} \otimes \phi_K}(X_K; \Q(t)^n) = 0$.
Hence $H_1^{\alpha_{r, p} \otimes \phi_K}(X_K; \Z[t^{\pm 1}]^n)$ is a torsion module, which completes the proof.
\end{proof}

\section{Proof of \texorpdfstring{\cref{thm_main}}{Theorem 1.2}} \label{sec_proof}

Now we prove \cref{thm_main}.
For the readers' convenience we recall the statement.%
\thmmain*

\begin{proof}
Let $K$ be an oriented knot in $S^3$ with $\Delta_K(t) \neq 1$, and let $V$ be the Seifert form associated with a Seifert surface~$F$.
Since $\Delta_K(t) \neq 1$, it follows from \cref{thm_Liv02} that there exists a positive integer $r$ such that $\Sigma_{K, r}$ is not a homology $3$-sphere.
Let $p$ be a prime factor of $\left| H_1(\Sigma_{K, r}; \Z) \right|$.
Since the homomorphisms
\[ H_1(X_K \setminus F; \Z[t^{\pm 1}]) \to H_1^{\phi_K}(X_K; \Z[t^{\pm 1}]) = H_K,~
H_K \otimes_{\Z[t^{\pm 1}]} \F_p[t^{\pm 1}] / (t^r-1) \to H_1(\Sigma_{K, r}; \F_p) \]
are surjective, there exists a simple closed curve in $X_K \setminus F$ with nontrivial image in $H_1(\Sigma_{K, r}; \F_p)$.
Let $A$ be such a simple closed curve in $X_K \setminus F$ that is unknotted in $S^3$.
Note that $A$ determines an element $[A] \in \ker \phi_K$ with nontrivial $\alpha_K^{r, p}([A])$ up to conjugation.

Use the satellite construction in \cref{subsec_satellite} to define a family of oriented knots $K_q$ in $S^3$ for primes $q \neq p$ by
\[ K_q = S = S(K, T(p, q) \sharp - T(p, q), A), \]
where $T(p, q)$ is the $(p, q)$-torus knot.
In the following we show that the family $\{ K_q \}_{q \neq p}$ of knots satisfies the conditions (i), (ii) and (iii) as in the statement.

Since $A$ is disjoint from $F$, it follows from \cref{lem_seifert} that for every $q$, $K_q$ has a Seifert form isomorphic to $V$.
Also, since $T(p, q) \sharp - T(p, q)$ is a ribbon knot, it follows from \cref{lem_concordance} that for every $q$, $K_q$ is ribbon concordant to $K$.
It remains to prove that for all primes $q, q' \neq p$, if $K_q \geq_{\top} K_{q'}$ then $q = q'$.
\bigskip

First we compute $\Delta_{K_q}^{r, p}(t)$ for a prime $q \neq p$.
Since $A$ is null-homologous, it follows from \cref{lem_infection} that $\rho_S \colon \pi_{K_q} \to \pi_K$ induces an isomorphism $\ol{\rho}_S \colon \pi_{K_q} / \pi_{K_q}^{(2)} \to \pi_K / \pi_K^{(2)}$.
By \cref{lem_MAP} we have
\begin{equation} \label{eq_p1}
\Delta_{K_q}^{r, p} \doteq \Delta_{K_q}^{\ol{\rho}_S^* \alpha_K^{r, p}} \doteq \Delta_{K_q}^{\alpha_K^{r, p} \circ \rho_S}.
\end{equation}
It follows from \cref{prop_satellite} and \cref{rem_resultant} that
\begin{equation} \label{eq_p2}
\Delta_{K_q}^{\alpha_K^{r, p} \circ \rho_S} \doteq \res(\Delta_{T(p, q) \sharp - T(p, q)}, f_{\alpha_K^{r, p}([A])}) \Delta_K^{r, p}.
\end{equation}
Here we have
\begin{equation} \label{eq_p3}
\Delta_{T(p, q) \sharp - T(p, q)}(t) \doteq \Delta_{T(p, q)}(t) \Delta_{-T(p, q)}(t) \doteq \Delta_{T(p, q)}(t)^2,
\end{equation}
\begin{equation} \label{eq_p4}
\Delta_{T(p, q)}(t) \doteq \Phi_{pq}(t) = \frac{(t^{pq}-1)(t-1)}{(t^p-1)(t^q-1)},
\end{equation}
where $\Phi_n(t)$ denotes the $n$-th cyclotomic polynomial, defined as
\[ \Phi_n(t) = \prod_{1 \leq k \leq n,~ (k, n) = 1} \left( t - e^{\frac{ 2 i \pi k}{n}} \right). \]
Using \cref{eq_p1,eq_p2,eq_p3,eq_p4}, conclude that
\begin{equation} \label{eq_p5}
\Delta_{K_q}^{r, p} \doteq \res(\Phi_{pq}^2, f_{\alpha_K^{r, p}([A])}) \Delta_K^{r, p} \doteq \res(f_{\alpha_K^{r, p}([A])}, \Phi_{pq})^2 \Delta_K^{r, p}.
\end{equation}

Now we show that $\res(\Phi_{pq}^2, f_{\alpha_K^{r, p}([A])})^2 = q^{m_q}$ for a positive integer~$m_q>0$.
Note that $\alpha_K^{r, p}([A])$ is a nontrivial permutation matrix, and since $\alpha_K^{r, p}|_{\ker \phi_K}$ factors through the $p$--group
\[ H_K \otimes_{\Z[t^{\pm 1}]} \F_p[t^{\pm 1}] / (t^r-1), \]
the order of $\alpha_K^{r, p}([A])$ is a nontrivial $p$--power.
Hence $f_{\alpha_K^{r, p}([A])}$ is a product of
$t^{p^k} - 1 = \prod_{i=1}^k \Phi_{p^i}$
for some positive integers $k$. Thus $\res(f_{\alpha_K^{r, p}([A])}, \Phi_{pq})$ is a nontrivial $q$--power, since $\res(\Phi_{p}, \Phi_{pq})=q$ and $\res(\Phi_{p^i}, \Phi_{pq})=1$ for $i > 1$, which follows from the following theorem: for positive integers $m$ and $n$ with $m < n$,
\[ \res (\Phi_m, \Phi_n) =
\begin{cases}
q^{\phi(m)} &\text{if $m ~|~ n$ and $\frac{n}{m}$ is a power of a prime $q$}, \\
1 &\text{otherwise},
\end{cases}
\]
where $\phi(n)$ is Euler's totient function~\cite{A70, D40, Leh30}, which is equal to $\deg \Phi_n$.
We conclude that \cref{eq_p5} implies that there exists a positive integer $m_q$ such that
\begin{equation} \label{eq_p6}
        \Delta_{K_q}^{r, p}(t) \doteq q^{m_q} \Delta_K^{r, p}(t).
\end{equation}

Now for arbitrary primes $q, q' \neq p$ suppose that $K_q \geq_{\top} K_{q'}$ with a homotopy ribbon concordance~$C$.
Since the Seifert forms of $K_q$ and $K_{q'}$ are isometric, the Alexander polynomials $\Delta_{K_q}$ and $\Delta_{K_{q'}}$ are equal.  By \cref{cor:Bl-isometric}, the Blanchfield pairings are isometric.  This implies that the submodule $G$ in the statement of \cref{thm_blanchfield} is trivial. But $G$ is the kernel of the inclusion induced map $j_* \colon H_1(X_{K_q};\Z[t^{\pm 1}]) \to H_1(X_C;\Z[t^{\pm 1}])$, so $j_*$ is injective.  Then \cref{prop_surjection}, applied with the abelianization representation $\alpha \colon \pi_C \to \GL(1,\Z) =\Z$ (or \cite[Proposition 3.1]{FP19}), implies that $j_*$ is surjective. On the other hand, the proof of \cref{thm_blanchfield} implies that the inclusion induced monomorphism $H_1(X_{K_{q'}};\Z[t^{\pm 1}]) \to H_1(X_C;\Z[t^{\pm 1}])$ identifies $H_1(X_{K_{q'}};\Z[t^{\pm 1}])$ with
\[j_*(G^{\perp}/G) = j_*(H_1(X_{K_q};\Z[t^{\pm 1}])) = H_1(X_{C};\Z[t^{\pm 1}]).\]
Therefore, the inclusion induced maps
\begin{equation}\label{eq:Alex-module-isos}
  H_1(X_{K_q};\Z[t^{\pm 1}]) \xrightarrow{\cong} H_1(X_{C};\Z[t^{\pm 1}]) \xleftarrow{\cong}  H_1(X_{K_{q'}};\Z[t^{\pm 1}]).
\end{equation}
are isomorphisms.
Now, for $T \in \{K_q,K_{q'},C\}$, we have a short exact sequence \[1 \to H_1(X_T;\Z[t^{\pm 1}]) \cong \pi_T^{(1)}/\pi_T^{(2)} \to \pi_T/\pi_T^{(2)} \to \Z \to 0\] that splits, so  $\pi_T/\pi_T^{(2)} \cong H_1(X_T;\Z[t^{\pm 1}]) \rtimes \Z$.

It follows that $\iota_{K_q} \colon \pi_{K_q} \to \pi_C$ and $\iota_{K_{q'}} \colon \pi_{K_{q'}} \to \pi_C$ induce isomorphisms $\ol{\iota}_{K_q} \colon \pi_{K_q} / \pi_{K_q}^{(2)} \to \pi_C / \pi_C^{(2)}$ and $\ol{\iota}_{K_{q'}} \colon \pi_{K_{q'}} / \pi_{K_{q'}}^{(2)} \to \pi_C / \pi_C^{(2)}$ respectively.
We write \[\alpha_C^{r, p} \colon \pi_C \to \GL(|\Gamma_{K_q}^{r, p}|, \Z)\] for the pullback of $\ol{\alpha}_{K_q}^{r, p} \circ \ol{\iota}_{K_q}^{-1}$ to $\pi_C$.

Note that $\alpha_C^{r, p}$ determines a representation $\alpha_C^{r, p} \circ \iota_{K_{q'}} \colon \pi_{K_{q'}} \to \GL(|\Gamma_{K_q}^{r, p}|, \Z)$ by composition with the inclusion map. The identification of Alexander modules in~\eqref{eq:Alex-module-isos} induces an isomorphism $\Gamma_{K_q}^{r, p} \xrightarrow{\cong} \Gamma_{K_{q'}}^{r, p}$, under which we may compare $\alpha_C^{r, p} \circ \iota_{K_{q'}}$ and $\alpha_{K_{q'}}^{r, p}$.  These representations of $\pi_{K_{q'}}$ may well differ, but the induced twisted Alexander polynomials agree up to units:
\begin{equation} \label{eq_p8}
\Delta_{K_{q'}}^{\alpha_C^{r, p} \circ \iota_{K_{q'}}} \doteq \Delta_{K_{q'}}^{(\ol{\iota}_{K_q}^{-1} \circ \ol{\iota}_{K_{q'}})^* \alpha_{K_q}^{r, p}} \doteq \Delta_{K_{q'}}^{r, p}.
\end{equation}
The first equation follows from the equality
$\alpha_C^{r, p} \circ \iota_{K_{q'}} = (\ol{\iota}_{K_q}^{-1} \circ \ol{\iota}_{K_{q'}})^* \alpha_{K_q}^{r, p}$ and the second one does from \cref{lem_MAP}.

Now we have all the tools we require to complete the proof that $K_q \geq_{\top} K_{q'}$ implies $q = q'$.
It follows from \cref{thm_division} that
\begin{equation} \label{eq_p7}
\Delta_{K_{q'}}^{\alpha_C^{r, p} \circ \iota_{K_{q'}}} ~|~ \Delta_{K_q}^{\alpha_C^{r, p} \circ \iota_{K_q}} \doteq \Delta_{K_q}^{r, p}.
\end{equation}
Combining \cref{eq_p8,eq_p7}, we obtain
\[\Delta_{K_{q'}}^{r, p} ~|~ \Delta_{K_q}^{r, p}.\]
It then follows from \cref{eq_p6} that
\[ q'^{m_{q'}} \Delta_K^{r, p} ~|~ q^{m_q} \Delta_K^{r, p}. \]
Since $m_q, m_{q'} > 0$, and since $\Delta_K^{r, p}(t) \neq 0$ by \cref{thm_non-vanishing}, we have $q = q'$.

Therefore, the family $\{ K_q \}_{q \neq p}$ of knots satisfies all the conditions (i), (ii) and (iii), and the proof of \cref{thm_main} is complete.
\end{proof}

\begin{appendix}
\section{Base changes for Blanchfield pairings}\label{sec:app}
The goal of this appendix is to prove \cref{prop:appmain,prop:appmain2},
which give a sufficient condition on a $\Z[t^{\pm 1}]$-algebra~$R$
for a Blanchfield pairing $\Bl^R$ to be definable in the same way as $\Bl$ is defined over~$\Z[t^{\pm 1}]$, such that
the analogue of \cref{thm_blanchfield} holds for $\Bl^R$.
Of course, a twisted Blanchfield pairing may be defined in a much more general situation than the one considered in this appendix (see e.g.~\cite{P16}).
The focus here is to choose $R$ such that
$\Bl^R_K$ is determined by $\Bl_K$, and $[J] \geq_S [K]$ implies the conclusion of
\cref{thm_blanchfield} for $\Bl^R_J$ and $\Bl^R_K$.

\cref{prop:appmain,prop:appmain2} are used at two places in this paper.
To prove \cref{cor:double}, which is about the homology of double branched covers of knots,
we set $R = \Z[t^{\pm 1}] / (t + 1)$.
To prove \cref{cor:sig}, which is about knot signatures, we set $R$ to be the localization of $\R[t^{\pm 1}]$
at $(\zeta)$, for $\zeta(t) = t^{-1} - 2\cos x + t$ a polynomial with nonreal roots on the unit circle.

Let us first establish some notation and definitions.
Let $\Xi \subset \Z[t^{\pm 1}]$ be the multiplicative set $\{p(t)\in \Z[t^{\pm 1}]\mid |p(1)| = 1\}$.
Let us call an element $a$ of a $\Z[t^{\pm 1}]$-module \emph{$\Xi$-torsion} if $q\cdot a = 0$ for some $q\in \Xi$,
and let us call a $\Z[t^{\pm 1}]$-module \emph{$\Xi$-torsion} if all of its elements are {$\Xi$-torsion}.
Let us call a $\Z[t^{\pm 1}]$-module $M$ \emph{$\Xi$-divisible} if for all $a\in M$ and $q\in \Xi$, there exists
$b\in M$ with $q\cdot b = a$.
Examples of $\Xi$-divisible modules are given by
$\Xi^{-1} \Z[t^{\pm 1}]$ and $\Xi^{-1} \Z[t^{\pm 1}] / \Z[t^{\pm 1}]$.

An \emph{admissible $\Z[t^{\pm 1}]$-algebra} is defined to be a
commutative unital $\Z[t^{\pm 1}]$-algebra $R$ with an involution, which we also denote by~$\ol{\,\cdot\,}$,
satisfying $\ol{p}\cdot\ol{r} = \ol{p\cdot r}$ for all $p\in \Z[t^{\pm 1}], r\in R$,
such that $R$ has no $(t-1)$-torsion and no $\Xi$-torsion.
Clearly, $\Z[t^{\pm 1}]$ is an admissible $\Z[t^{\pm 1}]$-algebra itself.
Further examples of such algebras are given by $\Z[t^{\pm 1}] / (q)$ for any polynomial $q \in \Z[t^{\pm 1}]$
that satisfies $\ol{(q)} = (q) \subset \Z[t^{\pm 1}]$ and $q(1) \neq 0$, and has no factor $p(t)\in \Z[t^{\pm 1}]$
with $|p(1)| = 1$. For instance, one may take $q$ as an $n$-th cyclotomic polynomial $\Phi_n$ for $n\geq 2$.
Yet further examples of admissible $\Z[t^{\pm 1}]$-algebras are given by localizations of $\Z[t^{\pm 1}]$
with respect to multiplicative sets $\Theta \subset \Z[t^{\pm 1}]$ with $0\not\in\Theta$
and $\ol{\Theta} = \Theta$.

For a $\Z[t^{\pm 1}]$-module $M$, we write $\Xi^{-1}M = M \otimes \Xi^{-1}\Z[t^{\pm 1}]$ for the localization of $M$ with respect to $\Xi$ (in this appendix, all tensor products are understood to be over $\Z[t^{\pm 1}]$).
Since $\Z[t^{\pm 1}]$ itself and $R$ are $\Xi$-torsion free, we have natural inclusions of modules (by sending $1$ to $1$)
\[
\Z[t^{\pm 1}] \subset \Xi^{-1}\Z[t^{\pm 1}] \subset \Q(t), \qquad
R \subset \Xi^{-1}R \subset Q(R),
\]
where we recall that $Q(R)$ is the localization of $R$ with respect to the multiplicative set of elements of $R$ that are not zero divisors. The homomorphisms $\iota\colon \Z[t^{\pm 1}] \to R$
and $\Xi^{-1}\Z[t^{\pm 1}] \to \Xi^{-1}R$ (both given by sending $1$ to $1$)
induce a homomorphism $\iota'\colon \Xi^{-1}\Z[t^{\pm 1}]/\Z[t^{\pm 1}] \to \Xi^{-1}R/R$.

\begin{prop}\label{prop:appmain}
Let $R$ be an admissible $\Z[t^{\pm 1}]$-algebra and $K$ a knot.
The four maps
\begin{multline*}
\begin{tikzcd}[sep=large,ampersand replacement=\&]
H_1(X_K; R) \ar[r] \& H_1(X_K, \partial X_K; R) \ar[r,"\text{PD}^{-1}"] \& \mbox{}
\end{tikzcd}\\
\begin{tikzcd}[sep=large,ampersand replacement=\&]
H^2(X_K; R) \ar[r,"\beta^{-1}"] \& H^1(X_K; Q/R) \ar[r,"\kappa"] \& \ohom_{R}( H_1(X_K; R), Q(R)/R)
\end{tikzcd}
\end{multline*}
defined analogously to \cref{eq:defBl}
are isomorphisms, and their composition
is the adjoint of a sesquilinear, Hermitian and nonsingular pairing
\[
\Bl^R_K\colon H_1(X_K; R) \times H_1(X_K; R) \to Q(R)/R.
\]
That pairing is determined by $\Bl_K$, satisfying
\[
\Bl^R_K(a\otimes r, b\otimes s) = r\cdot\ol{s}\cdot\iota'(\Bl_K(a,b))
\]
for all $a, b\in H_1(X_K; \Z[t^{\pm 1}])$ and $r,s\in R$
(here, we use that $\Bl_K(a,b) \in \Xi^{-1}\Z[t^{\pm 1}] \subset \Q(t)$).
\end{prop}

\begin{prop}\label{prop:appmain2}
Let $J$ and $K$ be knots, and let $G\subset H_1(X_J;\Z[t^{\pm 1}])$ be a submodule such that
the pairing on $G^{\bot} / G$ induced by $\Bl_J$ is isometric to $\Bl_K$.
\begin{enumerate}[label=\emph{(\roman*)}]
\item The pairing on $(G^{\bot}\otimes R) / (G \otimes R)$ induced by $\Bl^R_J$ is isometric to~$\Bl^R_K$.
\item $G^{\bot} \otimes R = (G\otimes R)^{\bot_R}$.
\end{enumerate}
\end{prop}
Note that in particular, if $\Bl_K$ is metabolic, then so is $\Bl^R_K$.
\Cref{prop:appmain} might not come as a surprise to the experts, but we thought it beneficial to include a detailed proof.
This proof is spread over two lemmas.
The first lemma contains the proofs that the four maps in \cref{prop:appmain} are isomorphisms,
and shows that $\Bl^R$ may be defined equivalently with target modules either $\Xi^{-1}R/R$ or $Q(R)/R$.

\begin{lem}\label{lem:changecoeff}
Let $R$ be an admissible $\Z[t^{\pm 1}]$-algebra and let $K$ be a knot.
\begin{enumerate}[label=\emph{(\roman*)}]
\item The natural maps $H_1(X_K; \Z[t^{\pm 1}])\otimes R \to H_1(X_K; R)$ and
$H_1(X_K, \partial X_K; \Z[t^{\pm 1}])\otimes R \to H_1(X_K, \partial X_K; R)$ are isomorphisms.
\item $H_1(X_K; R)$ and $H_1(X_K, \partial X_K; R)$ are $\Xi$-torsion modules $($in particular, $R$-torsion modules$)$,
and multiplication by $(1-t)$ acts invertibly on them.
\item The inclusion induced homomorphism $H_1(X_K; R) \to H_1(X_K, \partial X_K; R)$ is an isomorphism.
\item For $\Omega = Q(R)$ or $\Omega = \Xi^{-1}R$, there exists a unique isomorphism
\[
\delta\colon H^2(X_K; R)\to \ohom_R(H_1(X_K; R), \Omega/R),
\]
such that $\delta \circ\beta = \kappa$,
where $\beta\colon H^1(X_K; \Omega/R) \to H^2(X_K; R)$ is the Bockstein connecting homomorphism,
and $\kappa\colon H^1(X_K; \Omega/R) \to\ohom_R(H_1(X_K; R), \Omega/R)$ is the Kronecker evaluation.

\item For $\Omega = Q(R)$, $H^1(X_K;\Omega)$ is trivial, and $\delta = \kappa\circ\beta^{-1}$.

\item The following diagram commutes.
\[
\begin{tikzcd}
H^2(X_K; R) \ar[d,"\text{id}"]\ar[rr,"\delta"] && \ohom_{R}( H_1(X_K; R), \Xi^{-1}R/R) \ar[d,hook] \\
H^2(X_K; R) \ar[r,"\beta^{-1}"] & H^1(X_K; Q(R)/R) \ar[r,"\kappa"] & \ohom_{R}( H_1(X_K; R), Q(R)/R).
\end{tikzcd}
\]

\noindent
Here, the second vertical map is induced by the inclusion $\Xi^{-1}R/R \to Q(R)/R$.
\end{enumerate}
\end{lem}
\begin{proof}
(i)
The universal coefficient spectral sequence for homology groups~\cite[Theorem 10.90]{Ro09} yields a short exact sequence
\[
\begin{tikzcd}[column sep=small]
0 \ar[r] &
H_1(X_K; \Z[t^{\pm 1}])\otimes R \ar[r] &
H_1(X_K; R) \ar[r] & \Tor^{\Z[t^{\pm 1}]}_1(H_0(X_K; \Z[t^{\pm 1}]), R) \ar[r] & 0.
\end{tikzcd}
\]
Since $H_0(X_K; \Z[t^{\pm 1}]) \cong \Z[t^{\pm 1}] / (t-1)$, the $\Tor$-term above is isomorphic to
the $(t-1)$-torsion submodule of $R$, which is trivial by assumption.
So the map $H_1(X_K; \Z[t^{\pm 1}])\otimes R \to H_1(X_K; R)$ is an isomorphism.
For homology rel.\ $\partial X_K$, a similar but simpler argument can be made, since $H_0(X_K, \partial X_K; \Z[t^{\pm 1}])$ is trivial.
\smallskip

(ii) Consider $\Delta_K\in \Xi$.
Note that $\Delta_K$ annihilates $H_1(X_K; \Z[t^{\pm 1}])$,
and so $\Delta_K\cdot 1_R\in R$ annihilates $H_1(X_K; R)$ by (i).
The same argument works for homology rel.\ $\partial X_K$.

The polynomial $\Delta_K\in \Xi$ decomposes as $\Delta_K = (1-t)q(t) \pm 1$ for some $q(t)\in \Z[t^{\pm 1}]$.
Since $\Delta_K$ annihilates the modules in question, multiplication by $1-t$ is an isomorphism,
with inverse given by multiplication by $\mp q(t)$.
\smallskip

(iii) One computes
$H_0(\partial X_K; R) \cong H_1(\partial X_K; R) \cong R/(t-1)R$.
It follows that
the inclusion induced map $H_1(\partial X_K; R) \to H_1(X_K; R)$ and the connecting homomorphism $H_1(X_K, \partial X_K; R) \to H_0(\partial X; R)$ both vanish, since multiplication by $(1-t)$ is an isomorphism of $H_1(X_K, \partial X_K; R)$ and of $H_1(X_K; R)$ by (i) and (ii).
\smallskip

(iv)
Part of the Bockstein long exact sequence coming from $0 \to R\to \Omega\to \Omega/R\to 0$ is
\[
\begin{tikzcd}[column sep=small]
H^1(X_K; R) \ar[r] & H^1(X_K; \Omega) \ar[r,"\alpha"] & H^1(X_K; \Omega/R) \ar[r,"\beta"] & H^2(X_K; R) \ar[r] & H^2(X_K; \Omega).
\end{tikzcd}
\]
We have that
\[
H^2(X_K; \Omega) \cong H_1(X_K, \partial X_K; \Omega) \cong H_1(X_K, \partial X_K; R) \otimes \Omega
\]
by Poincaré duality and flatness of $\Omega$. The latter module
is trivial, since $H_1(X_K, \partial X_K; R)$ is $\Xi$-torsion by (ii).
It follows that $\beta$ is surjective.

The Kronecker evaluation $\kappa$ is part of the short exact sequence
\[
\begin{tikzcd}[column sep=small]
0 \ar[r] &
\Ext^1_R(H_0(X_K; R), \Omega/R) \ar[r,"\gamma"] & H^1(X_K; \Omega/R) \ar[r,"\kappa"] & \ohom_R(H_1(X_K; R), \Omega/R) \ar[r] & 0.
\end{tikzcd}
\]

Since $\beta$ and $\kappa$ are both surjective, to show existence and uniqueness of $\delta$ it suffices to
show that $\ker \beta = \ker \kappa$. For this, consider the following commutative diagram.
\[
\begin{tikzcd}
\Ext^1_R(H_0(X_K; R), \Omega) \ar[r]\ar[d] & \Ext^1_R(H_0(X_K; R), \Omega/R) \ar[d,"\gamma"] \\
H^1(X_K; \Omega) \ar[r,"\alpha"] & H^1(X_K; \Omega/R).
\end{tikzcd}
\]
Note that $\ohom(H_1(X_K; R), \Omega)$ vanishes, since $H_1(X_K; R)$ is $\Xi$-torsion by (ii).
So the left map in the diagram is an isomorphism. Moreover, $\Ext^2_R(H_0(X_K; R), R)$ is trivial,
since $H_0(X_K; R) = R/(1-t)R$. So the top map in the diagram is surjective.
It follows that
\[
\ker \kappa = \gamma(\Ext^1_R(H_0(X_K; R), \Omega/R)) = \alpha(H^1(X_K; \Omega)) = \ker \beta.
\]
\smallskip

(v)
As shown in (iv), $H^1(X_K; Q(R)) \cong \Ext^1_R(H_0(X_K; R), Q(R))$. Since $H_0(X_K; R) \cong R / (1-t)R$,
the $\Ext$-term is isomorphic to $Q(R) / ((1-t)Q(R))$, which is trivial since
$(1-t)$ is invertible in~$Q(R)$.
So $\beta$ is an isomorphism, and $\delta = \kappa\circ\beta^{-1}$.
\smallskip

(vi) This follows from naturality.
\end{proof}
The next lemma analyses the relationship between $\Bl$ and $\Bl^R$.
\begin{lem}\label{lem:ZtoR}
The following diagram of $\Z[t^{\pm 1}]$-modules commutes.
\newlength{\armove}\setlength{\armove}{3.5em}%
\newcommand{\ardl}{\ar[d,start anchor={[xshift=-\armove]south east},end anchor={[xshift=-\armove]north east}]}%
\newcommand{\ardr}{\ar[d,start anchor={[xshift=\armove]south west},end anchor={[xshift=\armove]north west}]}%
\noindent
\[
\begin{tikzcd}
[%
    ,/tikz/column 1/.append style={anchor=base east}
    ,/tikz/column 2/.append style={anchor=base west},sep=small
    ]
H_1(X_K; \Z[t^{\pm 1}]) \ardl\ar[r] &
H_1(X_K; R) \ardr
\\
H_1(X_K, \partial X_K; \Z[t^{\pm 1}]) \ardl\ar[r] &
H_1(X_K, \partial X_K; R) \ardr
\\
H^2(X_K; \Z[t^{\pm 1}]) \ardl \ar[r] &
H^2(X_K; R) \ardr
\\
\ohom_{\Z[t^{\pm 1}]}( H_1(X_K; \Z[t^{\pm 1}]), \Xi^{-1}\Z[t^{\pm 1}]/\Z[t^{\pm 1}])\ar[r] &
\ohom_{R}( H_1(X_K; R), \Xi^{-1}R/R)
\end{tikzcd}
\]
Here, the vertical maps are the isomorphisms given by inclusion, inverse of Poincaré duality and~$\delta$, respectively.
The first three horizontal maps are induced by $\iota$, and the fourth horizontal map is defined as composition of
\begin{align*}
& \begin{tikzcd}[ampersand replacement=\&,sep=scriptsize]
  \ohom_{\Z[t^{\pm 1}]}( H_1(X_K; \Z[t^{\pm 1}]), \Xi^{-1} \Z[t^{\pm 1}]/ \Z[t^{\pm 1}])
\ar[r, "\iota'_*"] \& \ohom_{\Z[t^{\pm 1}]}(H_1(X_K; \Z[t^{\pm 1}]), \Xi^{-1}R/R)
\ar[r, "\cong"] \& \mbox{}
\end{tikzcd} \\
& \begin{tikzcd}[ampersand replacement=\&,sep=scriptsize]
 \ohom_{\Z[t^{\pm 1}]}(H_1(X_K; \Z[t^{\pm 1}]), \hom_{R}(R, \Xi^{-1}R/R))
\ar[r, "\cong"] \& \ohom_R(H_1(X_K; \Z[t^{\pm 1}])\otimes R, \Xi^{-1}R/R)
\ar[r, "\cong"] \& \mbox{}
\end{tikzcd} \\
& \begin{tikzcd}[sep=scriptsize]
  \ohom_R(H_1(X_K; R), \Xi^{-1}R/R).
\end{tikzcd}
\end{align*}
\end{lem}

\begin{proof}
Commutativity of the diagram follows from the naturality of the maps (for the naturality of Poincaré duality, see e.g.~\cite[Section~A.3]{FNOP19}).
\end{proof}

\begin{proof}[Proof of \cref{prop:appmain}]
The first map is an isomorphism, as shown in \cref{lem:changecoeff}(iii). For the second map, see \cite[Section~A.3]{FNOP19}. The third and fourth map are discussed in \cref{lem:changecoeff}(iv) and (v).
That these maps are isomorphisms implies the nonsingularity of $\Bl^R_K$, whereas the sesquilinearity of $\Bl^R_K$ is automatic.
The formula relating $\Bl_K$ and $\Bl^R_K$ follows from \cref{lem:changecoeff}(vi) and \cref{lem:ZtoR}.
This formula also implies that $\Bl^R_K$ is Hermitian, since $\Bl_K$ is.
\end{proof}

For the proof of \cref{prop:appmain2}, we need several lemmas.
Much of the required homological algebra machinery was developed by Levine.
In particular, we will need the following statement.%

\begin{lem}[{\cite[Prop.~(3.5)]{Lev77}}]\label{lem:typeK}
If $M$ is a finitely generated $\Z$-torsion free $\Z[t^{\pm 1}]$-module on which multiplication with $1-t$ is an automorphism, then $M$ is of homological dimension $1$, i.e.\ there is a short exact sequence
\[
\begin{tikzcd}
0 \ar[r] & \Z[t^{\pm 1}]^n\ar[r] & \Z[t^{\pm 1}]^m \ar[r] & M \ar[r] & 0.
\end{tikzcd}
\]
for some $n,m \in \N$.\qed
\end{lem}

\begin{lem}\label{lem:homztorsionfree}
For every $\Z[t^{\pm 1}]$-module $M$,
$\hom_{\Z[t^{\pm 1}]}(M, \Xi^{-1}\Z[t^{\pm 1}]/\Z[t^{\pm 1}])$ is $\Z$-torsion free.
\end{lem}

\begin{proof}
Let a homomorphism $\alpha\colon M \to \Xi^{-1}\Z[t^{\pm 1}]/\Z[t^{\pm 1}]$ be given such that $m\alpha = 0$ for some nonzero $m\in\Z$. Let $a\in M$ be given. Write $\alpha(a) = n(t) / d(t) \in \Xi^{-1}\Z[t^{\pm 1}]/\Z[t^{\pm 1}]$, where $n(t)\in \Z[t^{\pm 1}]$ and $d(t)\in \Xi$ have no nonunital common divisors. Then $m\alpha(a) = mn(t)/d(t) = 0 \in \Xi^{-1}\Z[t^{\pm 1}]/\Z[t^{\pm 1}]$, which implies that $d(t)$ divides $mn(t)$ in $\Z[t^{\pm 1}]$. Because $\Z[t^{\pm 1}]$ is a UFD and $n(t)$ and $d(t)$ have no common divisor, it follows that $d(t)$ is an integer dividing $m$. But the only integers in $\Xi$ are~$\pm 1$.
It follows that $\alpha(a) = 0$. So we have established $\alpha = 0$.
\end{proof}

\begin{lem}\label{lem:alexztorsionfree}
The module $H_1(X_K;\Z[t^{\pm 1}])$ is $\Z$-torsion free for all knots $K$.
\end{lem}
\begin{proof}
The adjoint of the Blanchfield pairing is an isomorphism between the modules $H_1(X_K;\Z[t^{\pm 1}])$ and
$\ohom_{\Z[t^{\pm 1}]} (H_1(X_K; \Z[t^{\pm 1}]), \Q(t) / \Z[t^{\pm 1}])$.
The latter module is isomorphic to the module
$\ohom_{\Z[t^{\pm 1}]} (H_1(X_K; \Z[t^{\pm 1}]), \Xi^{-1} \Z[t^{\pm 1}] / \Z[t^{\pm 1}])$ by \cref{lem:changecoeff}.
That last module, in turn, is $\Z$\nobreakdash-torsion free by \cref{lem:homztorsionfree}.
\end{proof}

\begin{lem}\label{lem:tor}
Let $M$ be a finitely generated $\Z[t^{\pm 1}]$-module without $\Z$-torsion that is annihilated by some $q(t)\in \Xi$, and let $R$ be a $\Xi$-torsion free $\Z[t^{\pm 1}]$-algebra $R$ (possibly $\Z[t^{\pm 1}]$ itself).
\begin{enumerate}[label=\emph{(\roman*)}]
\item The $R$-module $M\otimes R$ has homological dimension 1.
\item For a $\Xi$-torsion free $R$-module $N$, the module $\Tor^R_1(M\otimes R,N)$ is trivial.
\item For a $\Xi$-divisible $R$-module $N$, the module $\Ext_R^1(M\otimes R,N)$ is trivial.
\end{enumerate}
\end{lem}

\begin{proof}
Let us first prove (i)--(iii) in the special case that $R = \Z[t^{\pm 1}]$.
\smallskip

\textit{(i) for $R = \Z[t^{\pm 1}]$.}
Similarly as in \cref{lem:changecoeff}(iii), one shows that
multiplication by $1-t$ is an isomorphism of $M$.
So $M$ is of homological dimension 1 by \cref{lem:typeK}.
\smallskip

\textit{(ii) for $R = \Z[t^{\pm 1}]$.}
Consider the short exact sequence
\[
\begin{tikzcd}
0 \ar[r] & N \ar[r,"q(t)"]& N \ar[r] & N /q(t) N \ar[r] & 0,
\end{tikzcd}
\]
which induces the exact sequence
\[
\begin{tikzcd}
\Tor^{\Z[t^{\pm 1}]}_2(M, N/q(t) N))
\ar[r] &
\Tor^{\Z[t^{\pm 1}]}_1(M, N)
\ar[r,"q(t)"] &
\Tor^{\Z[t^{\pm 1}]}_1(M, N).
\end{tikzcd}
\]
By (i), the first Tor term in the above sequence is trivial.
Hence multiplication by $q(t)$ is injective on $\Tor^{\Z[t^{\pm 1}]}_1(M, N)$. But it is also trivial, since
$q(t)$ annihilates~$M$. This implies that $\Tor^{\Z[t^{\pm 1}]}_1(M, N)$ is trivial.
\smallskip

\textit{(iii) for $R = \Z[t^{\pm 1}]$.}
The proof is dual to that of (ii), starting with the short exact sequence
\[
\begin{tikzcd}
0 \ar[r] & \{q\text{-torsion of }N\} \ar[r] & N \ar[r,"q(t)"]& N \ar[r] & 0.
\end{tikzcd}
\]
\smallskip

\textit{(i) for general $R$.} We have already shown that $M$ is of homological dimension $1$ over $\Z[t^{\pm 1}]$.
Tensoring the short exact sequence as in \cref{lem:typeK} with $R$ gives a short exact sequence
\[
\begin{tikzcd}
0 \ar[r] & R^n\ar[r] & R^m \ar[r] & M\otimes R \ar[r] & 0,
\end{tikzcd}
\]
showing that $M\otimes R$ is of homological dimension $1$ over $R$,
since $\tor^{\Z[t^{\pm 1}]}(M, R)$ is trivial by (ii).
\smallskip

Now, the proofs of (ii) and (iii) for general $R$ are analogous to the proofs in the special case $R = \Z[t^{\pm 1}]$.
\end{proof}

\begin{proof}[Proof of \cref{prop:appmain2}]
(i)
Consider the following short exact sequence of $\Z[t^{\pm 1}]$-modules:
\[
\begin{tikzcd}
0 \ar[r] & G \ar[r] & G^{\bot} \ar[r] & H_1(X_K; \Z[t^{\pm 1}]) \ar[r] & 0.
\end{tikzcd}
\]
The module $H_1(X_K; \Z[t^{\pm 1}])$ satisfies the hypotheses of \cref{lem:tor}(ii), since it
is $\Z$-torsion free by \cref{lem:alexztorsionfree},
and is annihilated by $\Delta_K\in \Z[t^{\pm 1}]$. Thus tensoring with $R$, and using \cref{lem:changecoeff}~(i), yields a short exact sequence of $R$-modules:

\[
\begin{tikzcd}
0 \ar[r] & G\otimes R \ar[r] & G^{\bot}\otimes R \ar[r] & H_1(X_K; R) \ar[r] & 0.
\end{tikzcd}
\]
Here, $G\otimes R$ and $G^{\bot}\otimes R$ carry pairings that are restrictions of $\Bl^R_J$,
while $H_1(X_K; R)$ carries the pairing $\Bl^R_K$. By \cref{prop:appmain},
the homomorphisms in the sequence respect those pairings.
This implies that the $\Bl^R_K$ is isometric to the pairing induced by $\Bl^R_J$
on $(G^{\bot}\otimes R) / (G\otimes R)$.
\smallskip

(ii)
There is a short exact sequence of $\Z[t^{\pm 1}]$-modules
\begin{equation}\label{eq:ses}
\begin{tikzcd}[column sep=small]
0 \ar[r] & G^{\bot} \ar[r] & H_1(X_J; \Z[t^{\pm 1}]) \ar[r] \to \ohom_{\Z[t^{\pm 1}]}(G, \Xi^{-1} \Z[t^{\pm 1}] / \Z[t^{\pm 1}]) \ar[r] & 0,
\end{tikzcd}
\end{equation}
where the second map is the inclusion, and the third map is the composition of the adjoint of $\Bl_J$ and the map induced by the inclusion $G\to H_1(X_J;\Z[t^{\pm 1}])$.
To see exactness, observe that $G^{\bot}$ is by definition the kernel of the third map, and the third map is surjective since
\[
\ext^1_{\Z[t^{\pm 1}]}(H_1(X_J;\Z[t^{\pm 1}])/G, \Xi^{-1} \Z[t^{\pm 1}] / \Z[t^{\pm 1}])
\]
is trivial by \cref{lem:tor}(iii).
The $\ohom$-module in \cref{eq:ses} is $\Z$-torsion free by \cref{lem:homztorsionfree}, and
it is annihilated by $\Delta_J$. So the module satisfies the hypothesis of \cref{lem:tor}(ii).
Hence tensoring with $R$ preserves exactness of~\cref{eq:ses}.
Consider now the following diagram of $R$-modules, in which the top row is \cref{eq:ses} tensored by $R$.
\[
\begin{tikzcd}[column sep=small]
0 \ar[r] & G^{\bot} \otimes R \ar[r]\ar[d] & H_1(X_J; R)\ar[r]\ar[d] & \ohom_{\Z[t^{\pm 1}]}(G, \Xi^{-1} \Z[t^{\pm 1}] / \Z[t^{\pm 1}]) \otimes R \ar[r]\ar[d] & 0 \\
0 \ar[r] & (G \otimes R)^{\bot_R} \ar[r]   & H_1(X_J; R)\ar[r]       & \ohom_R(G\otimes R, \Xi^{-1} R / R) \ar[r] & 0,
\end{tikzcd}
\]
The bottom row is the analog of \cref{eq:ses} over the base ring $R$, and its exactness is shown in the same way as for \cref{eq:ses}.
The first vertical map is just the inclusion, and the second map is the identity. The third map is
induced by $\iota'\colon \Xi^{-1}\Z[t^{\pm 1}]/\Z[t^{\pm 1}]\to\Xi^{-1}R/R$, using the Tensor-Hom-adjunction isomorphism
\begin{align*}
  \hom_R(G\otimes R, \Xi^{-1} R / R) &\cong \hom_{\Z[t^{\pm 1}]}(G, \Xi^{-1} R / R)
  \\  f &\mapsto \big(g \mapsto f(g \otimes 1)\big).
\end{align*}
Commutativity of the first square of the diagram is clear, and commutativity of the second square
follows from \cref{lem:ZtoR}.
It follows that the first vertical map is injective, and the third vertical map is surjective.
Our goal is to show that the first map is bijective;
for that, by an easy diagram chase as in the proof of the Five Lemma, it suffices to show that the third map is injective.

By \cref{lem:typeK}, $G$ admits a resolution
\[
\begin{tikzcd}
0 \ar[r] & \Z[t^{\pm 1}]^n \ar[r] & \Z[t^{\pm 1}]^m \ar[r] & G \ar[r] & 0
\end{tikzcd}
\]
for some $n,m\in\N$.
Since $\ext^1_{\Z[t^{\pm 1}]}(G, \Xi^{-1} \Z[t^{\pm 1}] / \Z[t^{\pm 1}]) = 0$ by \cref{lem:tor}(iii),
we obtain another short exact sequence by applying $\hom_{\Z[t^{\pm 1}]}(-,\Xi^{-1}\Z[t^{\pm 1}] / \Z[t^{\pm 1}])$:
\begin{equation}\label{eq:ses2}
\begin{array}{l}
\begin{tikzcd}[column sep=small]
0 \ar[r] &
\hom_{\Z[t^{\pm 1}]}(G,\Xi^{-1}\Z[t^{\pm 1}] / \Z[t^{\pm 1}]) \ar[r] &
B_m \ar[r] & B_n \ar[r] & 0,
\end{tikzcd}
\end{array}
\end{equation}
where we abbreviate $B_i = \hom_{\Z[t^{\pm 1}]}(\Z[t^{\pm 1}]^{\oplus i},\Xi^{-1}\Z[t^{\pm 1}] / \Z[t^{\pm 1}])$.
Let us show that $\tor_1^{\Z[t^{\pm 1}]}(R, B_n)$ is trivial.
Tensoring the short exact sequence
\[
\begin{tikzcd}
0  \ar[r] & \Z[t^{\pm 1}]  \ar[r] & \Xi^{-1}\Z[t^{\pm 1}]  \ar[r] & \Xi^{-1}\Z[t^{\pm 1}]/\Z[t^{\pm 1}] \ar[r] & 0
\end{tikzcd}
\]
with $R$ yields the exact sequence
\[
\begin{tikzcd}
\tor^{\Z[t^{\pm 1}]}_1(R, \Xi^{-1}\Z[t^{\pm 1}])  \ar[r] &
\tor^{\Z[t^{\pm 1}]}_1(R, \Xi^{-1}\Z[t^{\pm 1}] / \Z[t^{\pm 1}])  \ar[r] &
R \ar[r] & \Xi^{-1} R.
\end{tikzcd}
\]
In that sequence, the first module is trivial since localizations are flat, and the last map is injective, since $R$ is $\Xi$-torsion free. It follows that the second module is trivial. Since $B_i \cong (\Xi^{-1}\Z[t^{\pm 1}] / \Z[t^{\pm 1}])^{\oplus i}$,
we have that $\tor_1^{\Z[t^{\pm 1}]}(R, B_n) \cong \tor_1^{\Z[t^{\pm 1}]}(R, \Xi^{-1}\Z[t^{\pm 1}] / \Z[t^{\pm 1}])^{\oplus i}$ is trivial, as desired.

Hence tensoring the short exact sequence \cref{eq:ses2} with $R$ preserves exactness.
Part of that sequence forms the first row in the following diagram of $R$-modules:
\[
\begin{tikzcd}[column sep=small]
0 \ar[r] &
\hom_{\Z[t^{\pm 1}]}(G,\Xi^{-1}\Z[t^{\pm 1}] / \Z[t^{\pm 1}])\otimes R \ar[r]\ar[d] &
B_m \otimes R \ar[d] \\
0 \ar[r] &
\hom_{\Z[t^{\pm 1}]}(G,\Xi^{-1}R / R) \ar[r] &
\hom_{\Z[t^{\pm 1}]}(\Z[t^{\pm 1}]^m,\Xi^{-1}R / R)
\end{tikzcd}
\]
Note the second row, obtained from the resolution by applying $\hom_{\Z[t^{\pm 1}]}(-,\Xi^{-1}R / R)$, is also exact, the diagram commutes, and the second vertical map is an isomorphism.
This implies that the first vertical map is injective, concluding the proof.%
\end{proof}
\end{appendix}


\end{document}